\documentclass[a4paper,11pt]{scrarticle}
\usepackage[a4paper]{geometry}
\usepackage[utf8]{inputenc}

\usepackage[titletoc]{appendix}

\usepackage[intlimits]{amsmath}
\numberwithin{equation}{section}
\usepackage{amsfonts,amssymb}
\usepackage{xpatch,amsthm}
\usepackage{graphicx}
\usepackage{subcaption}

\theoremstyle{plain}
\newtheorem{theorem}{Theorem}\numberwithin{theorem}{section}
\newtheorem{lemma}[theorem]{Lemma}
\newtheorem{proposition}[theorem]{Proposition}
\newtheorem{corollary}[theorem]{Corollary}

\theoremstyle{definition}
\newtheorem{remark}[theorem]{Remark}
\newtheorem{definition}[theorem]{Definition}
\newtheorem{assumption}[theorem]{Assumption}

\newcommand{\R}{\mathbb{R}}
\newcommand{\N}{\mathbb{N}}
\newcommand*\di{\mathop{}\!\mathrm{d}}

\DeclareMathOperator{\diam}{diam}
\DeclareMathOperator{\tr}{tr}
\DeclareMathOperator{\supp}{supp}

\DeclareMathOperator{\sgn}{sgn}

\newcommand{\dual}[2]{\langle #1 , #2 \rangle}
\newcommand{\MM}{\mathcal{M}}
\newcommand{\KK}{\mathcal{K}}
\newcommand{\CC}{\mathcal{C}}
\newcommand{\NN}{\mathcal{N}}

\title{A Priori Error Analysis for an Optimal Control Problem Governed by a Variational Inequality of the Second Kind}

\author{C. Meyer\thanks{Faculty of Mathematics, Technische Universit\"{a}t Dortmund, 44227 Dortmund, Germany, e-mail:
\texttt{cmeyer@math.tu-dortmund.de}}
\and M. Weymuth\thanks{Institute for Mathematics and Computer-Based Simulation, Universit\"{a}t der Bundeswehr M\"{u}nchen, 85577 Neubiberg, Germany, e-mail:
\texttt{monika.weymuth@unibw.de}}}

\begin{document}

\maketitle

\begin{abstract}We consider an optimal control problem governed by an elliptic variational inequality of the second kind. The problem is discretized by linear finite elements for the state and a variational discrete approach for the control. 
Based on a quadratic growth condition we derive nearly optimal a priori error estimates.
Moreover, we establish second order sufficient optimality conditions that ensure a quadratic growth condition. 
These conditions are rather restrictive, but allow us to construct a one-dimensional locally optimal solution with 
reduced regularity, which serves as an exact solution for numerical experiments.\end{abstract}

\textbf{Keywords:} optimal control, elliptic variational inequalities of the second kind, finite elements, a priori error analysis\\

\textbf{Mathematics Subject Classification (2010):} 49M25, 65G99, 65K15, 65N30

\section{Introduction}
This paper is concerned with the following optimal control problem governed by a variational inequality of the second kind:

\begin{equation}\label{prob_P}\tag{P}
	\left.\begin{aligned}
	&\min_{(y, u) \in H_0^1(\Omega)\times L^2(\Omega)} J(y,u):= \frac{1}{2}\|y-y_d\|_{L^2(\Omega)}^2
	+\frac{\nu}{2}\|u  - u_d \|_{L^2(\Omega)}^2, \\
	&\text{s.t.}\quad \int_\Omega \nabla y \cdot \nabla (v-y)\ \mathrm dx+\|v\|_{L^1(\Omega)}-\|y\|_{L^1(\Omega)}\geq \langle u, v-y\rangle\quad \forall v\in H_0^1(\Omega) 
	\end{aligned}
	\right\}
\end{equation}

The precise assumptions on the quantities in \eqref{prob_P} will be given in Section \ref{sec:Preliminaries}.\\
Optimal control problems governed by variational inequalities play an important role in many applications such as problems in contact mechanics, phase separation or elastoplasticity. Therefore, these problems have been intensively investigated in the last years. Special techniques have been introduced for the characterization of local optima and numerical methods have been developed for various applications. However, discretization error estimates are rarely examined in the literature.\\
Our aim is to derive optimal a priori error estimates for the finite element discretization of \eqref{prob_P}. The field of a priori error analysis for optimal control problems governed by PDEs is often adressed in the literature. We refer e.g. to \cite{AradaCasasTroeltzsch2002, CasasTroeltzsch2002, DeckelnickHinze2007, Meyer2008, Roesch2004} and the references therein. However, to the authors' best knowledge, there is only one paper \cite{MeyerThoma2013} deriving quantitative error estimates for a finite element discretization of an optimal control problem governed by a variational inequality, namely in the special case of the obstacle problem.\\
The error analysis of optimal control problems governed by variational inequalities is
challenging since it needs a combination of optimal control theory and a priori finite element error analysis for variational inequalities. While error estimates  in the energy norm are well-known in the literature (see e.g. \cite{Falk1974}, \cite[Section~11]{AtkinsonHan2009}) there are still remaining questions concerning the behavior of the error in lower $L^p$-norms. For the error analysis of the problem class under consideration especially $L^2$-error estimates are of interest. However, an adaptation of the well-known Aubin-Nitsche trick to our variational inequality seems to be impossible in general due to a lack of regularity of the dual problem, see \cite{Natterer1976} and \cite{ChristofMeyer2018_2} for a discussion of the obstacle problem. A remedy to circumvent this difficulty is to employ $L^\infty$-error estimates 
 following the technique introduced in \cite{Nochetto1988}.\\
Another difficulty in the investigation of these optimal control problems is the fact that the control-to-state operator is in general  not G\^{a}teaux differentiable. 
    Therefore, the standard procedure for the derivation of necessary and sufficient optimality conditions based 
    on the adjoint calculus is not readily applicable and modified stationarity concepts 
    such as Clarke-, Bouligand-, Mordukhovich or strong stationarity have been introduced. 
Among these concepts the strong stationarity is the most rigorous one. It was first discussed by Mignot and Puel \cite{MignotPuel1984} for the obstacle problem and was further developed and an\-a\-lyzed for different kind of problems e.g. in \cite{HintermuellerSurowiec2011, HerzogMeyerWachsmuth2013, Wachsmuth2014, livia}. Moreover, based on strong stationarity second order sufficient conditions can be established which ensure a quadratic growth condition, see \cite{KunischWachsmuth2012, BetzMeyer2015}. 
    For variational inequalities of the second kind strong stationarity conditions are established in \cite{DelosReyesMeyer2016, christof, CMST2020}. 
    Second order sufficient conditions ensuring local optimality are however not yet known for this class of problems.\\
Under the assumption that the quadratic growth condition holds, we establish nearly optimal a priori error estimates for the control. An important ingredient for our error analysis is the $L^\infty$-error estimate for the state which is derived in the appendix
    following the lines of \cite{Nochetto1988}. 
As main result we obtain up to a logarithmic term a linear convergence rate for the $L^2$-norm of the control. This result is in agreement with the result of \cite{MeyerThoma2013} for the obstacle problem.
    In order to construct an example with an exact (local) solution fulfilling the growth condition required for our error analysis, 
    we derive fairly restrictive sufficient conditions ensuring a quadratic growth. Nevertheless, 
    these conditions are still weak enough to enable the construction of an exact solution, which only provides 
   a reduced regularity of the adjoint state.\\
The paper is structured as follows: In Section \ref{sec:Preliminaries} we clarify the notation as well as the precise assumptions on the quantities in \eqref{prob_P}. We also state some well-known results concerning the existence, uniqueness and regularity of solutions to \eqref{prob_P}.  Section \ref{sec:Discretization} is devoted to the finite element discretization of \eqref{prob_P}. The error analysis is presented in Section \ref{sec:error_analysis}. In Section \ref{sec:NSOC} necessary and sufficient optimality conditions are discussed. Finally in Section \ref{sec:1D_Ex} we construct a one-dimensional example with 
 reduced regularity and validate our theoretical results by numerical experiments in Section \ref{sec:numerics}. \ref{sec:Appendix} contains a regularization procedure for the variational inequality. Based on a technique introduced in \cite{Nochetto1988} we derive $L^\infty$-error estimates for the state which are an important ingredient for the proof of our convergence rates. Moreover, in \ref{sec:Reg_OCP} a regularized optimal control problem and its discretization are investigated.

\section{Preliminaries}\label{sec:Preliminaries}
\subsection{Notation and Problem Statement}

Throughout this work we use the standard notation $H_0^1(\Omega)$ and $W^{k,p}(\Omega)$, $k\in\N$, $1\leq p\leq \infty$ for the Sobolev spaces on a domain $\Omega\subset \R^d$, $d\geq 1$. We refer to \cite{Adams2003} for details of these spaces. As usual the dual of $H_0^1(\Omega)$ w.r.t. the $L^2$-inner product is denoted by $H^{-1}(\Omega)$ and the symbol $\langle \cdot, \cdot\rangle$ denotes the dual pairing between $H_0^1(\Omega)$ and $H^{-1}(\Omega)$. The $L^2$-scalar product on $\Omega$ is denoted by $(\cdot,\cdot)$.\\
$C>0$ denotes a constant which may vary at different occurences but is always independent of the relevant parameters such as mesh size $h$ or regularization parameter $\gamma$.\\
Moreover, we introduce the bilinear form $a: H_0^1(\Omega)\times H_0^1(\Omega)\to \R$ by
\begin{equation*}
a(y,v):=(\nabla y,\nabla v).
\end{equation*}

The coercivity constant of $a$ will be denoted by $c$, i.e.
\begin{equation}\label{coercivity}
a(v,v)\geq c\|v\|_{H^1(\Omega)}^2\quad \forall v\in H_0^1(\Omega).
\end{equation}

We impose the following assumptions on the data in \eqref{prob_P}:
\begin{itemize}
\item[i)] $\Omega\subset \R^d$ ($d=1,2,3$) is either a bounded interval or a polygonal/polyhedral bounded domain which is $W^{2,p}$-regular for $p<\infty$ if $d=2$ or $p\leq 6$ if $d=3$.
\item[ii)] The desired state satisfies $y_d\in L^2(\Omega)$ and $\nu>0$ is a fixed real number.
\item[iii)] The desired control $u_d$ is a function in $L^\infty(\Omega)$.
\end{itemize}

\begin{remark}
We call the domain $\Omega$ to be $W^{2,p}$-regular if the solution of the Laplace equation $-\Delta w=f$ with Dirichlet boundary condition satisfies $w\in W^{2,p}(\Omega)$ provided that $f\in L^p(\Omega)$. For $d=2$ the regularity assumption i) holds if the largest interior angle is less than or equal to $\pi/2$ (cf.  e.g. \cite[Section~2.7]{Grisvard1992}).\\
 If $d=3$ the situation gets more complicated. In this case we have two types of singularities, those at the edges and those at the vertices. In general the edge singularities are nastier than the vertex singularities. For a detailed discussion of edge and vertex singularities we refer to \cite{Grisvard1992, KozlovMazyaRossmann1997, KozlovMazyaRossmann2001, Mazya2010}.\\
Assumption i) is e.g. satisfied if $\theta<3\pi/2$ and $\Lambda^+>3/2$ (see \cite[Theorem~4.3.2]{Mazya2010}). Here $\theta:=\max_k\theta_k$, where $\theta_k$ denotes the interior dihedral angle at the edge $e_k$, and $\Lambda^+:=\min_k\Lambda_k^+$, where $\Lambda_k^+$ denotes the smallest positive eigenvalue of the Laplace Beltrami operator on the intersection of $\Omega$ and the unit sphere centered at the vertex $v_k$.
\end{remark}

\subsection{Known Results}
In the following we summarize some known results about the variational inequality and the optimal control problem \eqref{prob_P}. We start with an existence and uniqueness result.

\begin{lemma}\label{lemma_Lipschitz}
For every $u\in H^{-1}(\Omega)$ the variational inequality
\begin{equation}\label{VI}
\quad a(y,v-y)+\|v\|_{L^1(\Omega)}-\|y\|_{L^1(\Omega)}\geq \langle u, v-y\rangle\quad \forall v\in H_0^1(\Omega)
\end{equation}
has a unique solution $y\in H_0^1(\Omega)$. \\
Moreover, the associated solution operator $S: H^{-1}(\Omega)\to H_0^1(\Omega)$ mapping $u$ to $y$ is globally Lipschitz continuous with Lipschitz constant $L=1/c$ with $c$ as in \eqref{coercivity}.
\end{lemma}

The proof is standard and can be found e.g. in \cite{Barbu1984}.\\

Next we state an equivalent reformulation of \eqref{VI} by means of a complementarity-like system. For a  proof we refer to \cite{DelosReyes2011}.
\begin{lemma}\label{lem:slack}
A function $y\in H_0^1(\Omega)$ solves \eqref{VI}, iff there exists a $q\in L^\infty(\Omega)$ such that
\begin{gather}
a(y, v) + (q, v) =\langle u, v\rangle\quad \forall v \in H_0^{1}(\Omega) \label{eq:qpde}\\
q(x)y(x) =|y(x)|,\quad |q(x)|\leq 1\quad \text{a.e. in}\ \Omega.
\end{gather}
Hence, if $u\in L^p(\Omega)$, $p\in(1,\infty)$, then $y\in W^{2,p}(\Omega)$.
\end{lemma}

\begin{definition}\label{def:Q}
In view of \eqref{eq:qpde}, the slack variable $q$ is unique and depends Lipschitz continuously on $u$, 
when considered as an element of $H^{-1}(\Omega)$.
Thus, similarly to the solution operator $S$, we can introduce a globally Lipschitz continuous mapping 
$Q: H^{-1}(\Omega) \ni u \mapsto q\in H^{-1}(\Omega)$.
\end{definition}

\begin{remark}\label{rem_regularity}
Since in our optimal control problem \eqref{prob_P} the control satisfies $u\in L^2(\Omega)$, the minimal regularity of the solution of \eqref{VI} is $y\in H^2(\Omega)$. Moreover, due to the continuous embedding $H^2(\Omega)\hookrightarrow L^p(\Omega)$ for all $p\geq 2$ we have $y\in L^\infty(\Omega)$.
\end{remark}

Due to the continuity of $S$ and the weak lower semicontinuity of $J$ we have the following result which can be found e.g. in \cite{Barbu1984}: 
\begin{proposition}
There exists a globally optimal solution of \eqref{prob_P} which is in general not unique due to the nonlinearity of $S$.
\end{proposition}

Finally we state a regularity result for the control, which is essential for our error analysis in Section \ref{sec:error_analysis}.

\begin{proposition}\label{regularity_control}
Every locally optimal solution satisfies 
$\bar{u}\in L^{\kappa}(\Omega)$ with $\kappa = \infty$, if $d=1$, $\kappa < \infty$, if $d=2$, and 
$\kappa = 6$, if $d=3$.
\end{proposition}

\begin{proof}
We apply the regularization approach with penalty from \ref{sec:Reg_OCP} and
consider the following regularized problem
\begin{equation}\label{eq:optprobreg}
    \left.
	\begin{aligned}
	& \min_{(y_\gamma, u_\gamma) \in H_0^1(\Omega)\times L^2(\Omega)} 	
	J(y_\gamma,u_\gamma) + \|u_\gamma - \bar u\|_{L^2(\Omega)}^2
	,\\
	& \text{s.t.}\quad a(y_\gamma, v)+\frac{2}{\pi}\int_\Omega\arctan(\gamma y_\gamma) v\ \mathrm dx =\langle u_\gamma, v\rangle\quad \forall v\in H_0^1(\Omega),
	\end{aligned}
	\quad \right\}
\end{equation}
where $\gamma>0$. By Theorem \ref{theo_strong_conv_P_gamma} we know that there exists a sequence $\{\bar{u}_\gamma\}$ such that $\bar{u}_\gamma\to \bar{u}$ in $L^2(\Omega)$ 
as $\gamma \to \infty$.
Further, according to Theorem \ref{theo_reg_opt_system}, the solution $\bar{u}_\gamma$ satisfies the gradient equation, i.e., 
$\bar p_\gamma + \nu (\bar u_\gamma - u_d) + 2(\bar u_\gamma - \bar u) = 0$.
Since the sequence $\{\bar{p}_\gamma\}$ is bounded in $H_0^1(\Omega)$ by Lemma \ref{lemma_pgamma_bounded}, there exists a subsequence, again denoted by $\{\bar{p}_\gamma\}$, such that $\{\bar{p}_\gamma\}$ converges weakly to 
a $\bar{p}$ in $H_0^1(\Omega)$. 
Together with the convergence of $\bar u_\gamma$ to $\bar u$ this implies
$\bar u = u_d - \nu ^{-1} \bar p$. Due to Sobolev embeddings and $u_d \in L^\infty(\Omega)$
by assumption this yields the claim.
\end{proof}

\section{Finite Element Discretization}\label{sec:Discretization}

We will discretize \eqref{prob_P} with linear finite elements. For this purpose we introduce a family of meshes $\{\mathcal{T}_h\}$. The mesh $\mathcal{T}_h$ consists of open simplices $T$ (intervals, triangles, tetrahedra) such that
\begin{equation*}
\bar{\Omega} =\bigcup_{T\in\mathcal{T}_h}\bar{T}
\end{equation*}
 and the mesh width is defined by
\begin{equation*}
h:=\max_{T\in \mathcal{T}_h} h_T\quad\text{with}\ h_T:=\diam (T).
\end{equation*}

We assume that $\mathcal{T}_h$ is shape-regular and quasi-uniform in the sense of \cite{BrennerScott}.\\

For the discretization of \eqref{prob_P} we introduce the space of piecewise linear functions
\begin{equation*}
V_h:=\{v_h\in H_0^1(\Omega):\ v_h|_T\in \mathbb{P}_1(T)\ \forall T\in\mathcal{T}_h\},
\end{equation*}

where $\mathbb{P}_1$ denotes the space of polynomials of degree $\leq 1$. Then the variational discretization of \eqref{prob_P} is given by

\begin{equation}\label{prob_Ph}\tag{P$_h$}
	\left.\begin{aligned}
	& \min_{(y_h,u)\in V_h\times L^2(\Omega)} J(y_h,u):= \frac{1}{2}\|y_h-y_d\|_{L^2(\Omega)}^2+\frac{\nu}{2}\|u - u_d\|_{L^2(\Omega)}^2,\\
	&\text{s.t.}\ a( y_h,v_h-y_h)+\|v_h\|_{L^1(\Omega)}-\|y_h\|_{L^1(\Omega)}\geq \langle u, v_h-y_h\rangle\ \ \, \forall v_h\in V_h.
	\end{aligned}
	\right\}
\end{equation}

Standard arguments yield the following result:

\begin{lemma}
For all $u\in H^{-1}(\Omega)$ the discrete variational inequality
\begin{equation}\label{disc_VI}
a( y_h,v_h-y_h)+\|v_h\|_{L^1(\Omega)}-\|y_h\|_{L^1(\Omega)}\geq \langle u, v_h-y_h\rangle\ \ \, \forall v_h\in V_h
\end{equation}
has a unique solution $y_h\in V_h$. Moreover, the associated discrete solution operator $S_h: H^{-1}(\Omega)\to V_h\subset H_0^1(\Omega)$, $u\mapsto y_h$ is globally Lipschitz continuous.
\end{lemma}

Consequently we get the following existence result:
\begin{proposition}
Problem \eqref{prob_Ph} has a solution which is in general not unique.
\end{proposition}

It is important to note that there is no discretization of the control in \eqref{prob_Ph}. However, due to the following results it suffices to restrict the controls to the set $V_h$ in order to obtain a fully discrete optimization problem.

\begin{proposition}\label{prop:udisc}
If $\bar{u}_h$ is a locally optimal solution of \eqref{prob_Ph}, then $\bar{u}_h - u_d\in V_h$.
\end{proposition}

\begin{proof}
The proof is analogous to the continuous case in Proposition~\ref{regularity_control}. For convenience of the reader, 
we shortly sketch the arguments. 
We consider the discrete counterpart to \eqref{eq:optprobreg}:
\begin{equation}\label{prob_P_hgamma}
	\left.\begin{aligned}
	& \min_{(y_{h,\gamma}, u_{h,\gamma}) \in V_h\times L^2(\Omega)} J(y_{h,\gamma},u_{h,\gamma})+\|u_{h,\gamma}-\bar{u}_h\|_{L^2(\Omega)}^2\\
	& \text{s.t.}\quad a(y_{h,\gamma},v_h)+\frac{2}{\pi}(\arctan(\gamma y_{h,\gamma}), v_h) =\langle u_{h,\gamma}, v_h\rangle\quad \forall v_h\in V_h,
	\end{aligned}
	\right\}
\end{equation}
where $\gamma>0$. 
Now, based on Proposition~\ref{prop:properties_reg_disc_VI}, exactly the same arguments 
that lead to Theorem~\ref{theo_strong_conv_P_gamma} also give the existence of a sequence $\{\bar u_{h, \gamma}\}$ 
of locally optimal solutions to \eqref{prob_P_hgamma} that strongly converges in $L^2(\Omega)$ to $\bar u_h$ as $\gamma \to \infty$.

Analogously to Theorem \ref{theo_reg_opt_system} one derives the following optimality system for \eqref{prob_P_hgamma}:
For every locally optimal solution $\bar u_{h, \gamma}$ of \eqref{prob_P_hgamma} 
there exist $\bar p_{h, \gamma} \in V_h$ and $\bar\mu_{h, \gamma} \in L^2(\Omega)$ such that 
\begin{subequations}
\begin{align}
&a(\bar{y}_{h,\gamma}, v_h)+\frac{2}{\pi}(\arctan(\gamma\bar{y}_{h,\gamma}),v_h) =\langle\bar{u}_{h,\gamma}, v_h\rangle\quad\forall v_h\in V_h\label{reg_disc_opt_system_1}\\
&a(\bar{p}_{h,\gamma}, v_h)+\langle\bar{\mu}_{h,\gamma}, v_h\rangle= (\bar{y}_{h,\gamma} -y_d,v_h)\quad\forall v_h\in V_h\label{reg_disc_opt_system_2}\\
&
\bar{\mu}_{h,\gamma} = \frac{2\gamma}{\pi(1+\gamma^2 \bar{y}_{h,\gamma}^2)}\bar{p}_{h,\gamma}
\quad \text{a.e.\ in }\Omega
\nonumber\\
&(\bar{\mu}_{h,\gamma},\bar{p}_{h,\gamma})\geq 0\label{reg_disc_opt_system_3}\\
& 
\bar p_{h, \gamma} + \nu (\bar u_{h, \gamma} - u_d)+ 2(\bar{u}_{h,\gamma}-\bar{u}_h) = 0 
\quad \text{a.e.\ in }\Omega.
\label{reg_disc_opt_system_4}
\end{align}
\end{subequations}
Now, consider again the above sequence $\{\bar u_{h, \gamma}\}$ converging to $\bar u_h$. 
From Proposition~\ref{prop:properties_reg_disc_VI} we deduce the convergence of the associated states and thus, 
the boundedness of the latter in $H^1_0(\Omega)$. 
Then, using the sign condition in \eqref{reg_disc_opt_system_3},
we can argue exactly as in the proof of Lemma~\ref{lemma_pgamma_bounded} to show that 
the associated sequence of adjoint states $\{\bar{p}_{h,\gamma}\}$ is bounded in $H_0^1(\Omega)$.
To be more precise, we have
\begin{equation}\label{eq:boundphg}
    \|\bar p_{h, \gamma}\|_{H^1_0(\Omega)} \leq \frac{1}{c} \big(\|\bar y_{h,\gamma} \|_{L^2(\Omega)} 
    + \| y_{d} \|_{L^2(\Omega)} \big).
\end{equation}
Possibly after passing to a subsequence, 
this implies that $\bar{p}_{h,\gamma} \rightharpoonup \bar{p}_h$ in $H^1_0(\Omega)$ as $\gamma \to \infty$. 
Since $V_h$ is a closed subspace of $H^1_0(\Omega)$, we obtain $\bar p_h \in V_h$.
Therefore, the convergence of $\bar{u}_{h,\gamma}$ to $\bar u_h$ in combination with the gradient equation in 
\eqref{reg_disc_opt_system_4} gives $\bar{u}_h - u_d=-\nu^{-1}\bar{p}_h\in V_h$ as claimed.
\end{proof} 

\begin{corollary}
    If $u_d$ in the objective in \eqref{prob_Ph} 
    is replaced by a function in $V_h$, denoted by $u_{d,h}$ (e.g.\ a suitable quasi-interpolation), 
    then Proposition~\ref{prop:udisc} shows that every local solution of \eqref{prob_Ph} is indeed an element of $V_h$.
\end{corollary}

\begin{remark}\label{rem:bounduhH1}
    The above proof also shows an estimate for $\bar u_h$ in $H^1_0(\Omega)$.
    By the weak convergence of $\bar p_{h, \gamma}$ and the weak lower semicontinuity of norms, 
    the bound in \eqref{eq:boundphg} carries over to $\bar p_h$ and thus, the discrete gradient equation gives
    \begin{equation*}
        \|\bar u_h - u_d\|_{H^1_0(\Omega)} \leq \frac{1}{\nu c} \big(\|\bar y_{h} \|_{L^2(\Omega)} 
        + \| y_{d} \|_{L^2(\Omega)} \big),
    \end{equation*}
    where we used the strong convergence of the states by Proposition~\ref{prop:properties_reg_disc_VI}.
\end{remark}

\section{Error Analysis}\label{sec:error_analysis}
This section is devoted to the derivation of nearly optimal a priori error estimates for the control. The proof is an adaptation of the technique introduced in \cite{MeyerThoma2013}. Its most important ingredients are a quadratic growth condition and an $L^\infty$-error estimate for the state which is established in \ref{sec:Appendix}. It is worth noting that we do not use the discrete maximum principle as e.g. in \cite{Baiocchi1977, Nitsche1977, MeyerThoma2013} and therefore, the triangulation is not required to be acute.\\

 In order to simplify the notation we introduce the reduced functional
\begin{equation*}
f: L^2(\Omega)\to \R,\quad  f(u):=J(S(u),u)
\end{equation*}  
as well as the discrete reduced functional 
\begin{equation*}
f_h: L^2(\Omega)\to\R,\quad f_h(u):=J(S_h(u),u).
\end{equation*}
 
Let $\bar{u}\in L^2(\Omega)$ be a fixed local optimum of \eqref{prob_P}. For the derivation of a priori error estimates the following assumption is crucial.
 
 \begin{assumption}[Quadratic Growth Condition]\label{ass:QGC}
There are $\epsilon,\ \delta>0$ such that
\begin{equation*}
f(\bar{u})\leq f(u)-\delta \|u-\bar{u}\|_{L^2(\Omega)}^2\quad \forall u\in B_\epsilon(\bar{u}),
\end{equation*}
where $B_\epsilon(\bar{u}):=\{u\in L^2(\Omega) : \|u-\bar{u}\|_{L^2(\Omega)}\leq \epsilon\}$.
 \end{assumption}

\begin{remark}
For the obstacle problem Assumption \ref{ass:QGC} holds if $\bar{u}$ satisfies some second-order sufficient optimality conditions (cf. \cite{KunischWachsmuth2012}). However, for variational inequalities of the second kind results with respect to sufficient optimality conditions and local quadratic growth are unknown. In Theorem \ref{theo:sec_order_condition} we establish second-order sufficient optimality conditions for our problem under the assumption that the set $\mathcal{M}:=\{x\in\Omega : y(x)=0, \nabla y(x)\neq 0\}$ is empty.
\end{remark}

\begin{lemma}\label{lemma_strong_convergence}
Suppose that $\bar{u}$ satisfies Assumption \ref{ass:QGC}. Then there is a sequence $\{\bar{u}_h\}$ of locally optimal solutions to \eqref{prob_Ph} with $\bar{u}_h\to \bar{u}$ in $L^2(\Omega)$ as $h\to 0$.
\end{lemma}

\begin{proof}
    Based on Lemma~\ref{lem:convH1} proven in the appendix the following proof is standard and similar to
    the proof of Theorem~\ref{theo_strong_conv_P_gamma} (see also \cite[Lemma~5.5]{MeyerThoma2013}, 
    where an analogous result for the optimal control of the obstacle problem is proven).
    Nevertheless, for later purpose and for convenience of the reader, we sketch the arguments.
    Following the classical localization argument from \cite{CasasTroeltzsch2002} we define the following 
    localized discrete problems:
    \begin{equation}\tag{P$_{h, \varepsilon}$}\label{eq:Pheps}
        \min_{u\in B_\varepsilon(\bar u)} \; f_h(u),
    \end{equation}
    where $B_\varepsilon(\bar u)$ is the closed $L^2$-ball from Assumption~\ref{ass:QGC}. 
    By standard arguments the above problem admits a globally optimal solution for every $h>0$, denoted by 
    $\bar u_h$. Due to the constraint this sequence is bounded in $L^2(\Omega)$ and consequently admits 
    a weakly convergent subsequence with limit $\tilde u\in L^2(\Omega)$, 
    which, by compact embedding, converges strongly in $H^{-1}(\Omega)$. 
    By Lemma~\ref{lem:convH1} the associated states $\bar y_h := S(\bar u_h)$ converge strongly to $\tilde y := S(\tilde u)$.
    The weak lower semicontinuity of the objective along with the isolated local optimality of $\bar u$ implies $\tilde u = \bar u$.
    Moreover, the Tikhonov term in the objective yields the norm convergence of $\bar u_h$ so that 
    $\bar u_h \to \bar u$ in $L^2(\Omega)$. This implies that $\bar u_h$ is in the interior of $B_\varepsilon(\bar u)$ 
    for $h$ sufficiently small and therefore, $\bar u_h$ is a local solution of \eqref{prob_Ph}.
\end{proof}

In the following $\{\bar u_h\}$ always refers to the sequence from Lemma~\ref{lemma_strong_convergence}. 
The above proof shows that the sequence of discrete states $\{\bar y_h\}$ is bounded in $H^1_0(\Omega)$. 
Thus, Remark~\ref{rem:bounduhH1} leads to the following

\begin{lemma}\label{lemma_boundedness}
The sequence $\{\bar{u}_h\}$ is bounded in $H^1(\Omega)$.
\end{lemma}

Now we are in the position to state our main result.

\begin{theorem}[Convergence Rate]\label{conv_rates}
Suppose that $\bar{u}$ satisfies Assumption \ref{ass:QGC}. Then there exists a constant $C>0$ such that, for $h>0$ sufficiently small,
\begin{equation*}
\|\bar{u}-\bar{u}_h\|_{L^2(\Omega)}\leq 
\begin{cases}Ch|\log(h)|, &\text{if}\ d=1\\
Ch^{1-\epsilon}|\log(h)|^{\frac{1}{2}}, &\text{if}\ d=2\\
Ch^{\frac{3}{4}}|\log(h)|^{\frac{1}{2}}, & \text{if}\ d=3,\end{cases}
\end{equation*}
where $\epsilon>0$.
\end{theorem}

\begin{proof}
The proof follows the lines of \cite[Theorem~5.8]{MeyerThoma2013}.
As seen in the proof of Lemma~\ref{lemma_strong_convergence}, $\bar u_h$ is a global optimal solution of \eqref{eq:Pheps} 
and therefore,
\begin{equation}\label{fu}
f_h(\bar{u}_h)\leq f_h(\bar{u}).
\end{equation}
Moreover, for $h$ sufficiently small, we have $\bar{u}_h\in B_\epsilon(\bar{u})$.
Therefore, Assumption \ref{ass:QGC} and \eqref{fu} imply
\begin{align}\label{eq_1}
\delta\|\bar{u}_h-\bar{u}\|_{L^2(\Omega)}^2&\leq f(\bar{u}_h)-f_h(\bar{u}_h)+f_h(\bar{u})-f(\bar{u})+f_h(\bar{u}_h)-f_h(\bar{u})\nonumber\\
&\leq \left| f(\bar{u}_h)-f_h(\bar{u}_h)\right|+\left|f(\bar{u})-f_h(\bar{u})\right|.
\end{align}
For the first term in \eqref{eq_1}, we obtain
\begin{align}\label{eq_2}
&\left| f(\bar{u}_h)-f_h(\bar{u}_h)\right|=\frac{1}{2}\left| \|S(\bar{u}_h)-y_d\|_{L^2(\Omega)}^2-\|S_h(\bar{u}_h)-S(\bar{u}_h)+S(\bar{u}_h)-y_d\|_{L^2(\Omega)}^2\right|\nonumber\\
&\qquad\leq \frac{1}{2}\|S_h(\bar{u}_h)-S(\bar{u}_h)\|_{L^2(\Omega)}^2+ \|S_h(\bar{u}_h)-S(\bar{u}_h)\|_{L^2(\Omega)}\|S(\bar{u}_h)-y_d\|_{L^2(\Omega)}.
\end{align}
By Lemma~\ref{lemma_boundedness} the sequence $\{\bar u_h\}$ is bounded in $H^1(\Omega)$.
Therefore, due to the continuous embeddings $H^1(\Omega)\hookrightarrow L^\infty(\Omega)$, if $d=1$, 
and $H^1(\Omega)\hookrightarrow L^p(\Omega)$ for all $p<\infty$, if $d=2$, respectively, for all $p\leq 6$, if $d=3$,
Theorem \ref{theo_Nochetto2} leads to
\begin{equation}\label{eq_3}
\|S_h(\bar{u}_h)-S(\bar{u}_h)\|_{L^2(\Omega)}\leq \begin{cases}C(h|\log(h)|)^2, &\text{if}\ d=1\\
Ch^{2(1-\epsilon)}|\log(h)|, & \text{if}\ d=2\\
Ch^{\frac{3}{2}}|\log(h)|, & \text{if}\ d=3.\end{cases}
\end{equation}
Invoking again the boundedness of $\{\bar{u}_h\}$ in $H^1(\Omega)$
and the Lipschitz continuity of the solution operator $S$ (cf. Lemma \ref{lemma_Lipschitz}) imply that 
$\|S(\bar{u}_h)-y_d\|_{L^2(\Omega)}$ is bounded. Hence, by \eqref{eq_2} and \eqref{eq_3} we finally get
\begin{equation*}
\left| f(\bar{u}_h)-f_h(\bar{u}_h)\right|\leq \begin{cases}C(h|\log(h)|)^2, & \text{if}\ d=1\\
Ch^{2(1-\epsilon)}|\log(h)|, & \text{if}\ d=2\\
Ch^{\frac{3}{2}}|\log(h)|, & \text{if}\ d=3.\end{cases}
\end{equation*}
Applying the same arguments to the second term of \eqref{eq_1} completes the proof.
\end{proof}

\begin{remark}
In the three-dimensional case our theory only leads to a convergence rate of $3/4$
(neglecting the logarithmic terms).
The proof shows that this result would change,
if $L^2$-error estimates known for elliptic equations would also hold for the variational inequality \eqref{VI}.  
However, an adaptation of the well-known Nitsche trick from elliptic PDEs to our problem 
is unknown so far and, in view of the negative results for the obstacle problem, see \cite{ChristofMeyer2018_2}, 
also very unlikely to hold, cf.\ also \cite[Remark~2.5]{Nochetto1988} in this context.

Moreover, it is also worth noting that the $W^{2,p}$-regularity of the domain 
is crucial for the proof of the convergence rate for $d\geq 2$.
\end{remark}

\section{Necessary and Sufficient Optimality Conditions}\label{sec:NSOC}

The aim of the following two sections is to show that the convergence rates of Theorem~\ref{conv_rates} 
are not that bad as they appear at first glance. To this end we will construct an exact (locally optimal) solution, which 
only provides a reduced regularity, i.e., the adjoint state is no more regular than $H^{{3/2}-\varepsilon}(\Omega)$.
The crucial aspect in this context is to find a solution that satisfies the quadratic growth condition in Assumption~\ref{ass:QGC}.
Usually such a growth condition is ensured by second-order sufficient optimality conditions and we will 
follow the same approach here, too. However, due to the lack of differentiability of the control-to-state map 
the derivation of optimality conditions for \eqref{prob_P} is all but standard.
To be more precise, 
since the solution operator $S$ of the variational inequality \eqref{VI} is in general not G\^{a}teaux differentiable, 
the standard adjoint approach for the derivation of optimality conditions cannot be applied.

There are essentially two different alternative approaches to resolve this issue. The first is based on regularization, where 
optimality conditions are obtained by taking the limit in the regularized optimality system, 
see e.g.\ \cite{DelosReyes2011}. 
A drawback of this approach is that the passage to the limit in general leads to a loss of information and consequently 
the optimality conditions obtained in this way are not rigorous enough for our purposes.
Therefore, we follow the second approach and use strong stationarity conditions 
based on differentiability properties of the control-to-state operator $S$ from \cite{ChristofMeyer2018}, see also \cite[Theorem~5.2.15]{christof}.
More precisely, using the Hadamard directional differentiability of $S$, one can follow the lines of 
\cite{Mignot1976, MignotPuel1984} to derive an optimality system, which is equivalent to the purely primal 
optimality condition $f'(\bar u;h) \geq 0$ for all $h \in L^2(\Omega)$.
For the problem under consideration this has been carried out in detail in \cite{CMST2020}. 
The optimality conditions obtained in this way read as follows:

\begin{theorem}[Strong Stationarity, {\cite[Theorem~3.10]{CMST2020}}]\label{theo:strong_stationarity}
    Let $\bar u \in L^2(\Omega)$ be locally optimal and denote the associated state by $\bar y = S(\bar u)$. 
    We impose the following assumptions on $\bar y$:
    \begin{itemize}
        \item \emph{(Regularity)} It holds $\bar y  \in C^1(\Omega)$.
        \item \emph{(Structure of the Active Set)}
        There exists a set $\CC \subseteq \partial \{\bar y  \neq 0\} \cup \partial \Omega$ such that
        \begin{enumerate} 
            \item $ \mathcal{C}$ is closed and has $H^1(\R^d)$-capacity zero,
            \item  $( \partial \{\bar y  \neq 0\}  \cup \partial \Omega) \setminus \mathcal{C}$ 
            is a (strong) $(d-1)$-dimensional Lipschitz submanifold of $\mathbb{R}^d$,
            \item the sets 
            \begin{equation*}
                \qquad\qquad \mathcal{N}_+ := \{\nabla \bar y = 0\} \cap \partial \{\bar y > 0\} \setminus \mathcal{C}, \quad
                \mathcal{N}_- :=  \{\nabla \bar y = 0\} \cap \partial \{\bar y < 0\}  \setminus \mathcal{C} 
            \end{equation*}
            are relatively open in $( \partial \{\bar y  \neq 0\}  \cup \partial \Omega )\setminus \mathcal{C}$.
        \end{enumerate}
    \end{itemize}
    Define the set 
    \begin{equation}\label{eq:defM}
        \mathcal{M}:=\{x\in\Omega: \bar y(x) = 0,\ \nabla \bar y(x)\neq 0\}
    \end{equation}
    (which is well-defined due to the regularity assumption on $\bar y$) as well as the Hilbert space
    \begin{equation*}
        W_{\bar y}:=\left\{z\in H_0^1(\Omega): \int_\mathcal{M}\frac{(\tr z)^2}{\|\nabla \bar y\|}\mathrm ds<\infty \right\}, 
    \end{equation*}
    endowed with the scalar product
    \begin{equation*}
        (z_1,z_2)_{W_{\bar y}}:=a(z_1,z_2) + \int_\mathcal{M}2\frac{(\tr z_1)(\tr z_2)}{\|\nabla \bar y\|}\di s, 
    \end{equation*}
    where $\tr: H^1_0(\Omega) \to L^2(\MM)$ denotes the trace on $\MM$.
    Moreover, we introduce the convex cone $\KK(\bar y)$ by
    \begin{align*}
        \mathcal{K}(\bar y) :=\{v\in W_{\bar y}: \
        |v| = \bar q \, v \text{ a.e., where } y = 0, \;
        \tr(v^\pm)= 0\ \text{a.e. on}\ \mathcal{N}_\mp\},
    \end{align*}
    where $\bar q = Q(\bar u)\in L^\infty(\Omega)$ is the slack variable associated with $\bar y$, see Lemma~\ref{lem:slack}.

    Then there exist an \emph{adjoint state} $\bar p\in W_{\bar y}$ and a \emph{multiplier} $\bar \mu \in W_{\bar y}^*$ such that 
    \begin{subequations}\label{eq:sstatlasso}
    \begin{gather}
         \bar p + \nu(\bar u - u_d) = 0 \quad \text{a.e.\ in  } \Omega, \label{eq:lassograd}\\
         a(\bar p,  v)
        + 2 \int_{\MM} \frac{(\tr \bar p)(\tr v)}{\|\nabla \bar y \|} \, \di s
         = (\bar y - y_d, v)  - \dual{\bar \mu}{v}_{W_{\bar{y}}}, \quad \forall\,v \in W_{\bar y}, \label{eq:lassAD}\\
         \bar p\in \KK(\bar y), \quad \dual{\bar \mu}{v}_{ W_{\bar{y}}} \geq 0 \quad \forall\, v\in \KK(\bar y). \label{eq:lassoSlack}
    \end{gather}    
    \end{subequations}
\end{theorem}
 
Some words concerning this result are in order:

\begin{remark}~
\,
    \begin{itemize}
        \item Recall that a subset $\mathcal{N}\subset \R^d$ is a $(d-1)$-dimensional strong Lipschitz submanifold of $\R^d$,
        if $\mathcal{N}$ is  locally the graph of a Lipschitz function defined on $\R^{d-1}$ (cf. \cite{Walter1976}). 
        Moreover, note that 
        since $\bar y$ is assumed to be continuously differentiable, the implicit function theorem implies that $\MM$
        is a $(d-1)$-dimensional $C^1$-submanifold of $\mathbb{R}^d$.
        Furthermore, since $\mathcal{N}_+$ and $ \mathcal{N}_-$ are relatively open subsets of 
        $\partial \{\bar y  \neq 0\} \setminus \mathcal{C}$, they are themselves strong $(d-1)$-dimensional Lipschitz 
        submanifolds of $\mathbb{R}^d$.
        Consequently $\mathcal{M}$ and $\mathcal{N}_\pm$ are strong $(d-1)$-dimensional Lipschitz submanifolds
        and therefore, traces on these sets are well-defined.
        \item  The assumptions on the regularity of $\bar y$ and the structure of the active set are needed 
        for the directional differentiability of the control-to-state map, see \cite{ChristofMeyer2018, christof}.
        Conditions different from the ones in Theorem~\ref{theo:strong_stationarity} 
        guaranteeing the directional differentiability of $S$ can be found in \cite{DelosReyesMeyer2016, hs18}. 
        While the former are even more restrictive than the assumptions used here, 
        the latter deal with the rather abstract assumption that the feasible set associated with the dual problem 
        of \eqref{VI} is polyhedric, which is in general wrong and may be hard to verify in practice, see \cite{GW19}.
        Differentiability results that go without any additional assumptions are not known so far, 
        at least to the best of our knowledge. In this respect, the VI under consideration behaves different from the obstacle 
        problem, where the directional differentiability holds without any further assumptions, see \cite{Mignot1976}.
        \item As already mentioned, the stationarity conditions in \eqref{eq:sstatlasso} are equivalent to purely primal 
        optimality conditions (see \cite{CMST2020} for details). They can thus be regarded as the most rigorous 
        stationarity conditions, in particular they are sharper than the ones in \cite{DelosReyes2011}, 
        the latter obtained via regularization.
    \end{itemize}
\end{remark} 
 
An adaptation of the technique introduced in \cite{KunischWachsmuth2012} for the obstacle problem leads to the following second-order sufficient optimality conditions:

\begin{theorem}[Second-Order Sufficient Conditions]\label{theo:sec_order_condition}
Assume that $\bar{u}\in L^2(\Omega)$ is such that $\bar y = S(\bar u)$ satisfies the regularity assumption
and the structural assumption on the active set in Theorem~\ref{theo:strong_stationarity}.
Moreover, suppose that the set $\MM$ from \eqref{eq:defM} is empty (so that $W_{\bar y} = H^1_0(\Omega)$) 
and that $\bar{p}\in H_0^1(\Omega)$ and $\bar{\mu}\in H^{-1}(\Omega)$ exist such that 
the strong stationarity conditions in \eqref{eq:sstatlasso} are fulfilled.
Furthermore, assume that there exists an $\epsilon>0$ so that, for every $u\in L^2(\Omega)$ 
with $\|u-\bar{u}\|_{L^2(\Omega)}\leq \epsilon$, there holds
\begin{equation}\label{eq:signcond}
\langle\bar{\mu}, S(u)-\bar{y}\rangle-\langle Q(u)-\bar{q}, \bar{p}\rangle\geq 0.
\end{equation}
Then $\bar{u}$ fulfills the quadratic growth condition from Assumption \ref{ass:QGC}.
\end{theorem}

\begin{proof}
The proof starts with the classical argument by contradiction.
Let us suppose that Assumption \ref{ass:QGC} is not fulfilled. Then there exist sequences $u_k\to \bar{u}$ in $L^2(\Omega)$, $y_k=S(u_k)$, $q_k=Q(u_k)$ such that
\begin{equation}\label{eq:contra}
J(\bar{y},\bar{u})+\frac{1}{k}\|u_k-\bar{u}\|_{L^2(\Omega)}^2> J(y_k,u_k) \quad \forall\, k\in \N.
\end{equation}
We abbreviate $t_k:=\|u_k-\bar{u}\|_{L^2(\Omega)}$.
By testing the equations for $(\bar y, \bar q,  \bar u)$ and $(y_k, q_k, u_k)$ with $\bar p$
and using \eqref{eq:lassograd} and \eqref{eq:lassAD} (with $\MM = \emptyset$), we arrive at
\begin{equation*}
\begin{aligned}
    0 &= \dual{-\Delta(y_k - \bar y) + (q_k - \bar q) - (u_k - \bar u)}{\bar p}\\
    &= (\bar y - y_d, y_k - \bar y) - \dual{\bar \mu}{y_k - \bar y} + (q_k -\bar q, \bar p) + \nu (\bar u -u_d , u_k -\bar u).
\end{aligned}
\end{equation*}
Together with \eqref{eq:contra}, it follows that
\begin{equation*}
\begin{aligned}
\frac{t_k^2}{k} &> J(y_k,u_k)-J(\bar{y},\bar{u})\\
&= (\bar y - y_d , y_k - \bar y) + \nu (\bar u -u_d , u_k - \bar u) +
\frac{1}{2}\|\bar{y}-y_k\|_{L^2(\Omega)}^2+\frac{\nu}{2}\|\bar{u}-u_k\|_{L^2(\Omega)}^2\\
&= \frac{1}{2}\|\bar{y}-y_k\|_{L^2(\Omega)}^2+\frac{\nu}{2}\|\bar{u}-u_k\|_{L^2(\Omega)}^2
+\langle\bar{\mu},y_k-\bar{y}\rangle-\langle q_k-\bar{q}, \bar{p}\rangle.
\end{aligned}
\end{equation*}
Dividing by $t_k^2 = \|\bar{u}-u_k\|_{L^2(\Omega)}^2$ yields
\begin{equation}\label{positivity_req}
\frac{1}{t_k^2} \Big(\langle \bar{\mu}, y_k - \bar y\rangle-\langle q_k - \bar q,\bar{p}\rangle \Big)
< \frac{1}{k}-\frac{1}{2}\left[ \frac{\|y_k - \bar y\|_{L^2(\Omega)}^2}{t_k^2} + \nu \right] 
\leq  \frac{1}{k}-\frac{1}{2}\,\nu 
\end{equation}
Due to \eqref{eq:signcond} the left-hand side is non-negative for $k\in \N$ sufficiently large and thus, 
$0<\nu \leq \frac{1}{k} \to 0$, which yields the desired contradiction.
\end{proof}

\newpage
\begin{remark}~
\,
    \begin{itemize}
        \item The above result is of rather theoretical interest, since the sign condition in \eqref{eq:signcond} 
        is hard to verify in practice. However, it allows us to construct an exact solution that satisfies the 
        quadratic growth condition in the next section.
        \item One can substantially generalize the above analysis following the lines of \cite{KunischWachsmuth2012}. 
        In particular, one can allow for more general (not necessarily strictly convex) objective functionals by using the Hadamard directional differentiability of $S$. 
        However, to keep the discussion concise and since the result is of limited practical use as seen above, 
        we restrict to the precise form of our objective.
        \item The essential difference to the proof of \cite[Theorem~2.2]{KunischWachsmuth2012} is that 
        we require the sign condition in \eqref{eq:signcond} for the sum of the dual pairings involving the multipliers 
        $\bar p$ and $\bar \mu$, while in \cite{KunischWachsmuth2012}, conditions are established, which ensure 
        the sign condition for each dual pairing individually. For the construction of our exact solution however, 
        it is essential to treat both addends together. 
    \end{itemize}
\end{remark}

\section{A One-Dimensional Example with Reduced Regularity}\label{sec:1D_Ex}

In this section we construct an exact locally optimal solution to a one-dimensional example for \eqref{prob_P}, 
which fulfills the sufficient optimality conditions from Theorem~\ref{theo:sec_order_condition} (and thus also the quadratic 
growth condition) and, at the same time, provides only reduced regularity. 
More precisely, the adjoint state is only an element of 
$H^{{3/2}-\varepsilon}(\Omega)$, $\varepsilon > 0$.

For this purpose let $\Omega:=(-1,1)\subset\R$ and assume that $\mathcal{M}=\emptyset$ (according to the assumptions in 
Theorem~\ref{theo:sec_order_condition}). 
Then the strong stationarity system of Theorem \ref{theo:strong_stationarity} simplifies to
 \begin{subequations}\label{strong_stat_1D}
 \begin{gather}
a(\bar{y}, v) + (\bar{q}, v) =\langle \bar{u}, v\rangle\quad \forall v \in H_0^{1}(-1,1)\label{eq1}\\
\bar{q}(x)\bar{y}(x)=|\bar{y}(x)|,\quad |\bar{q}(x)|\leq 1\quad \text{a.e. in}\ (-1,1)\label{eq2}\\
a(\bar{p},v)= (\bar{y}-y_d, v)- \langle \bar{\mu},v\rangle \quad \forall v\in H_0^1(-1,1)\label{eq3}\\
\bar{p}\in \mathcal{K}(\bar{y}), \quad \langle \bar{\mu},v\rangle\geq 0 \quad \forall v\in \mathcal{K}(\bar{y})\label{eq4}\\
\bar{p}+\nu (\bar{u} - u_d) = 0.\label{eq5}
\end{gather}
\end{subequations}

Our aim is to construct a locally optimal solution of \eqref{prob_P} with minimal regularity. We will first define a solution of the strong stationarity system \eqref{strong_stat_1D}  and afterwards verify that the second-order sufficient conditions of Theorem \ref{theo:sec_order_condition} are satisfied. We start with defining the optimal state by
\begin{equation*}
\bar{y}(x):=\begin{cases}
88\, \alpha\, x^3+156\,\alpha\, x^2+82\,\alpha \,x+14\,\alpha , & x\in[-1,-0.5)\\
16\,\alpha\, x^4, & x\in [-0.5,0.5]\\
-88\,\alpha\, x^3+156\,\alpha \,x^2-82\,\alpha\, x+14\,\alpha , & x\in (0.5,1]
\end{cases}
\end{equation*}
with a constant $\alpha >0$. We observe that $\bar{y}\in C^2([-1,1])$ such that the regularity assumption in 
Theorem~\ref{theo:strong_stationarity} is fulfilled. 
Further we observe that $\bar{y}(x)>0$ for all $x\in (-1, 1)\setminus \{0\}$ and $\bar{y}(0)=\bar{y}'(0)=0$, 
which implies that $\mathcal{M}=\emptyset$ as required.
Moreover, as sets of zero $H^1(\R^d)$-capacity are empty in one dimension, we have $\NN_+ = \{0\}$, 
while $\NN_- = \emptyset$. Note that $\NN_+ = \partial\{\bar y \neq 0\} = \{0\}$ and thus, $\NN_+$ is relatively open.
Hence, all assumptions of Theorem~\ref{theo:strong_stationarity} are satisfied. Moreover, if we set 
\begin{equation*}
\bar{q}\equiv 1,
\end{equation*}
then the complementarity system \eqref{eq2} holds. Together with $\NN_+ = \{0\}$ and $\NN_- = \emptyset$ 
this choice implies that 
\begin{equation}\label{eq:Kbsp}
\mathcal{K}(\bar{y})=\{v\in H_0^1(-1,1): v(0)\geq 0\}.
\end{equation}

For the adjoint state we choose
\begin{equation*}
\bar{p}(x):=\begin{cases}
-(m+\beta)x^2-mx+\beta,& x\in [-1,0]\\
-(m+\beta)x^2+mx+\beta,& x\in (0,1],
\end{cases}
\end{equation*}
where $\beta=\bar{p}(0)>0$ and $ m>0$. 
Note that $\bar{p}$ has a kink located at zero and is therefore not twice weakly differentiable.
In addition, as $\bar p(0) = \beta > 0$, we have $\bar{p}\in\mathcal{K}(\bar{y})$ as required in \eqref{eq4}.
Furthermore,  since $\bar{p}$ is smooth in $(-1,0)$ and in $(0,1)$, it follows from integration by parts that
\begin{align*}
\int_{-1}^1 \bar{p}'v' \di x
= - 2m v(0) + \int_{-1}^1 2(m+\beta) v\di x .
\end{align*}
Therefore, if we set 
\begin{equation}\label{eq:yd}
    \bar{\mu} : = 2 m \,\delta_0\, \in C([-1,1])^* \hookrightarrow H^{-1}(-1,1)
    \quad \text{and}\quad 
    y_d := \bar{y}- 2(m+\beta),
\end{equation}
then the adjoint equation in \eqref{eq3} is fulfilled.
Herein, $\delta_0$ denotes the Dirac measure centered at zero. Note that, due to \eqref{eq:Kbsp}, 
\begin{equation*}
    \dual{\bar\mu}{v} = 2 m \, v(0) \geq 0 \quad \forall\, v\in \KK(\bar y)
\end{equation*}
so that the second condition in \eqref{eq4} is also satisfied. 
The optimal control is obtained by \eqref{eq1}, i.e.
\begin{equation*}
\bar{u}(x) = - \bar y''(x) + \bar q(x) =\begin{cases}
-528 \,\alpha\,  x - 312\, \alpha  +1, & x\in[-1,-0.5)\\
-192 \,\alpha\,  x^2 +1, & x\in[-0.5,0.5]\\
528\, \alpha\,  x - 312 \,\alpha  + 1, & x\in (0.5,1].
\end{cases}
\end{equation*} 
Finally, to fulfill the gradient equation \eqref{eq5}, we define
 \begin{equation}\label{eq:ud}
 u_d(x):=\frac{1}{\nu}\,\bar{p}(x)+\bar{u}(x).
 \end{equation}
With this setting the strong stationarity conditions \eqref{strong_stat_1D} are satisfied. It remains to check that the sufficient conditions are satisfied. 

For this purpose we need the following auxiliary result:

\begin{lemma}\label{lem:ylip}
    Let $u \in L^1(-1,1)$ be given. Then $y = S(u)$ is Lipschitz continuous on $[-1,1]$ with Lipschitz constant 
    $L_u = \|u\|_{L^1(-1,1)} + 2$.
\end{lemma}

\begin{proof}
    Since $y'' = q - u$ with the slack variable $q \in L^\infty(-1,1)$ from Lemma~\ref{lem:slack}, we find that $y \in W^{2,1}(-1,1)$. 
    Therefore, $y\in C^1([-1,1])$ and, since $y(-1) = y(1) = 0$, Rolle's theorem implies that there is a point 
    $\tilde x \in (-1,1)$ such that $y'(\tilde x) = 0$.
    
    Now, let $x_1, x_2 \in [-1, 1]$ be arbitrary. Since $y' \in C([-1,1])$, there is a point $\bar x\in [-1,1]$ with 
    $\|y'\|_{L^\infty(-1,1)} = |y'(\bar x)|$. Consequently we obtain
    \begin{equation*}
    \begin{aligned}
        |y(x_2) - y(x_1)|
        & \leq  \|y'\|_{L^\infty(-1,1)}\, |x_2-x_1| 
        = \big|y'(\bar x) - y'(\tilde x)\big|\, |x_2-x_1|\\
        &\leq \int_{-1}^1 |y''(x)| \di x\, |x_2-x_1|
        \leq \big( \|u\|_{L^1(-1,1)} + 2\big)\, |x_2-x_1|,
    \end{aligned}
    \end{equation*}
    where we used that $|q| \leq 1$ a.e.\ in $(-1,1)$ for the last estimate.
\end{proof}

With this result at hand we are in the position to verify the conditions of Theorem \ref{theo:sec_order_condition}. 
As we have already checked the assumptions of Theorem~\ref{theo:strong_stationarity}, verified that $\MM = \emptyset$
and seen that the strong stationarity system is fulfilled, we only need to prove that 
there is an $\epsilon > 0$ such that
\begin{equation}\label{eq1_sec_order_conditions}
\langle\bar{\mu}, S(u) -\bar{y}\rangle-\langle Q(u)-\bar{q}, \bar{p}\rangle\geq 0
\end{equation}
holds for all $u\in L^2(-1,1)$ with $\|u - \bar u\|_{L^2(-1,1)}\leq \epsilon$.
For this purpose let $u\in L^2(-1,1)$ be arbitrary and set $y := S(u)$ and $q:=Q(u)$.
By construction we have
\begin{equation}\label{eq2_sec_order_conditions}
\langle\bar{\mu}, y-\bar{y}\rangle-\langle q-\bar{q}, \bar{p}\rangle = 2m \,y(0)+\int_{-1}^1 \bar{p}(x)[1-q(x)]\, \di x.
\end{equation}
We distinguish the following two cases:
\begin{itemize}
\item[i)] $y(0)\geq 0$:\\
For our example it holds $\bar{p}(x)\geq 0$ for all $x\in [-1,1]$. Therefore, since $|q(x)|\leq 1$ a.e. in $(-1,1)$
by Lemma~\ref{lem:slack}, the right-hand side of \eqref{eq2_sec_order_conditions} is non-negative and \eqref{eq1_sec_order_conditions} is satisfied.
\item[ii)] $y(0) < 0$:\\
Since $y$ is Lipschitz continuous by Lemma~\ref{lem:ylip}, it is negative in a neighborhood of zero 
and, as we know the Lipschitz constant, we can estimate this neighborhood: Assume that $\delta \in [-1,1]$ 
is the root of $y$, which is closest to $0$. Then Lemma~\ref{lem:ylip} gives 
\begin{equation*}
   |y(0)| = |y(0) - y(\delta)| \leq L_u |\delta|.
\end{equation*}

Hence, $y$ is negative on the interval $(-\delta, \delta)$ with $\delta\geq |y(0)|/L_u$ and thus, by Lemma~\ref{lem:slack}, 
we have $q =-1$ on $(-\delta, \delta)$. Moreover, by construction, $\bar{p}(x)\geq \beta$ for all $x\in (-\delta, \delta)$
provided that $\delta$ is sufficiently small. Therefore, we obtain the estimate
\begin{align*}
\int_{-1}^1\bar{p}(x)[1-q(x)]\,\mathrm dx + 2m\, y(0) 
&\geq \int_{-\delta}^{\delta}\bar{p}(x)[1-q(x)]\,\mathrm dx - 2m\, |y(0)|\\
& \geq 4\,\beta\,\delta -2m\, |y(0)| \geq 2\Big(2\,\frac{\beta}{L_u}-m\Big) |y(0)|.
\end{align*}
 Thus, if 
\begin{equation}
    \beta\geq \frac{1}{2}\,L_u\,m,
\end{equation}
then the right-hand side of \eqref{eq2_sec_order_conditions} is non-negative.
\end{itemize}

In order to estimate the Lipschitz constant we use the triangle and H\"older's inequality to obtain
\begin{equation*}
    \|u\|_{L^1(-1,1)}
    \leq \|\bar u\|_{L^1(-1,1)} + \sqrt{2} \| u - \bar u\|_{L^2(-1,1)}.
\end{equation*}
Hence, for all $u \in L^2(\Omega)$ with $\| u - \bar u\|_{L^2(-1,1)}\leq \epsilon$ there holds
\begin{equation*}
    L_u = \|u\|_{L^1(-1,1)} + 2 \leq \|\bar u\|_{L^1(-1,1)} + \sqrt{2} \,\epsilon + 2.
\end{equation*}
Therefore, if 
\begin{equation}\label{eq:lipconst}
    \beta\geq \frac{1}{2}\, m \big(\|\bar u\|_{L^1(-1,1)} + \sqrt{2} \,\epsilon + 2\big),
\end{equation}
then condition \eqref{eq1_sec_order_conditions} is satisfied.
Now, it is easy to see that $\bar u \geq 0$, if $\alpha \leq 11/528$, so that one obtains 
\begin{equation*}
    \|\bar u\|_{L^1(-1,1)} = \int_{-1}^1 \bar u(x) \,\di x = 68\,\alpha + 2.
\end{equation*}
Finally, we end up with the following conditions for the parameters $\alpha$, $\beta$, and $m$:
\begin{equation}\label{eq:paramcond}
    \alpha \leq \frac{11}{528} \quad \text{and}\quad 
    \beta\geq \frac{1}{2}\, m \big(68\,\alpha + \sqrt{2} \,\epsilon + 4\big).
\end{equation}
Herein $\epsilon > 0$ can be chosen. The larger it is chosen, the bigger the neighborhood of local optimality of $\bar u$
becomes.

We observe that the weak derivatives of $\bar{p}$ and $\bar{u}$ are piecewise smooth functions, 
see also Figure~\ref{fig:exact_sol}.
\begin{figure}[h]
\centering
\begin{subfigure}[b]{0.32\textwidth}
        \includegraphics[width=\textwidth]{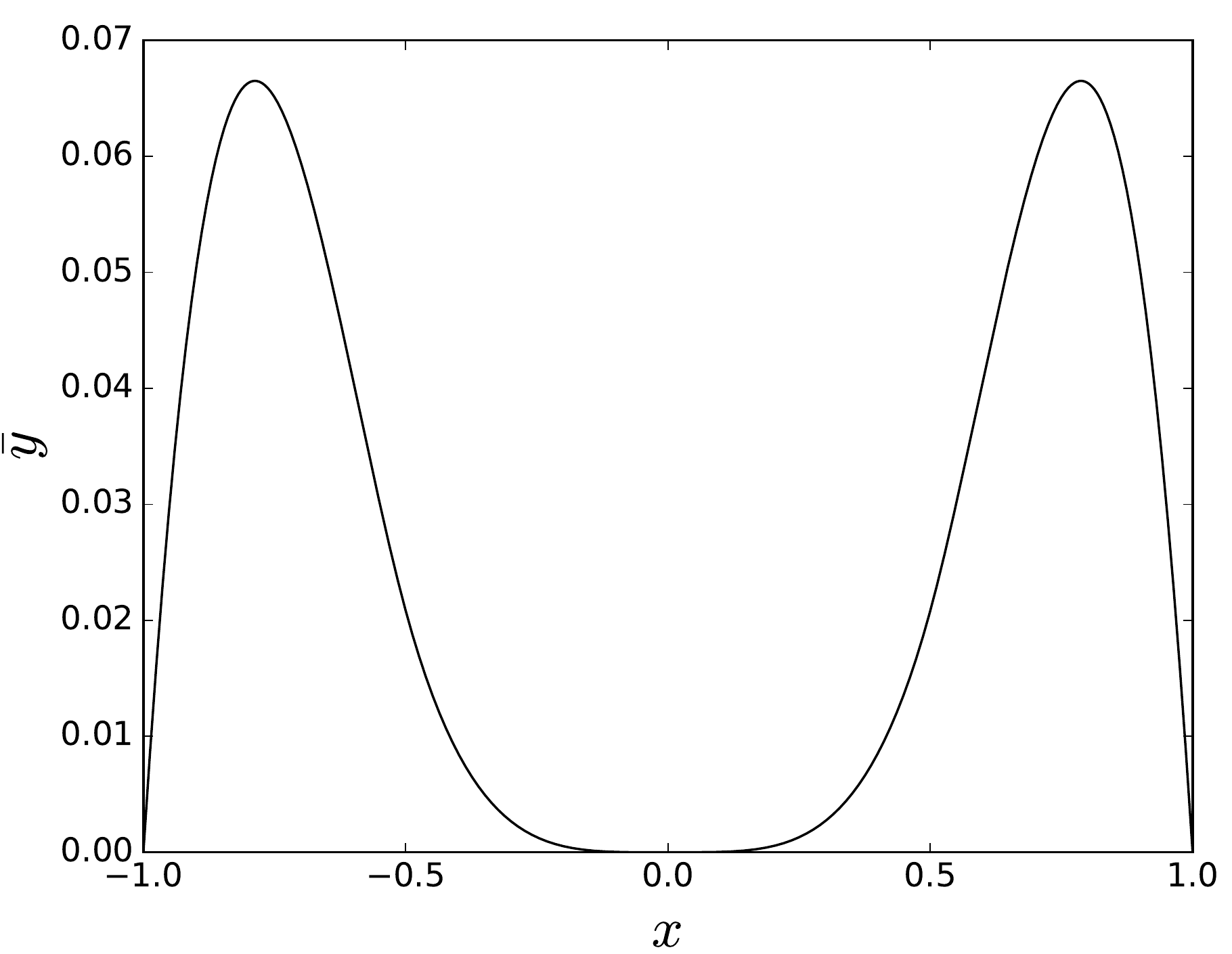}
       
    \end{subfigure}
    \
    \begin{subfigure}[b]{0.32\textwidth}
        \includegraphics[width=\textwidth]{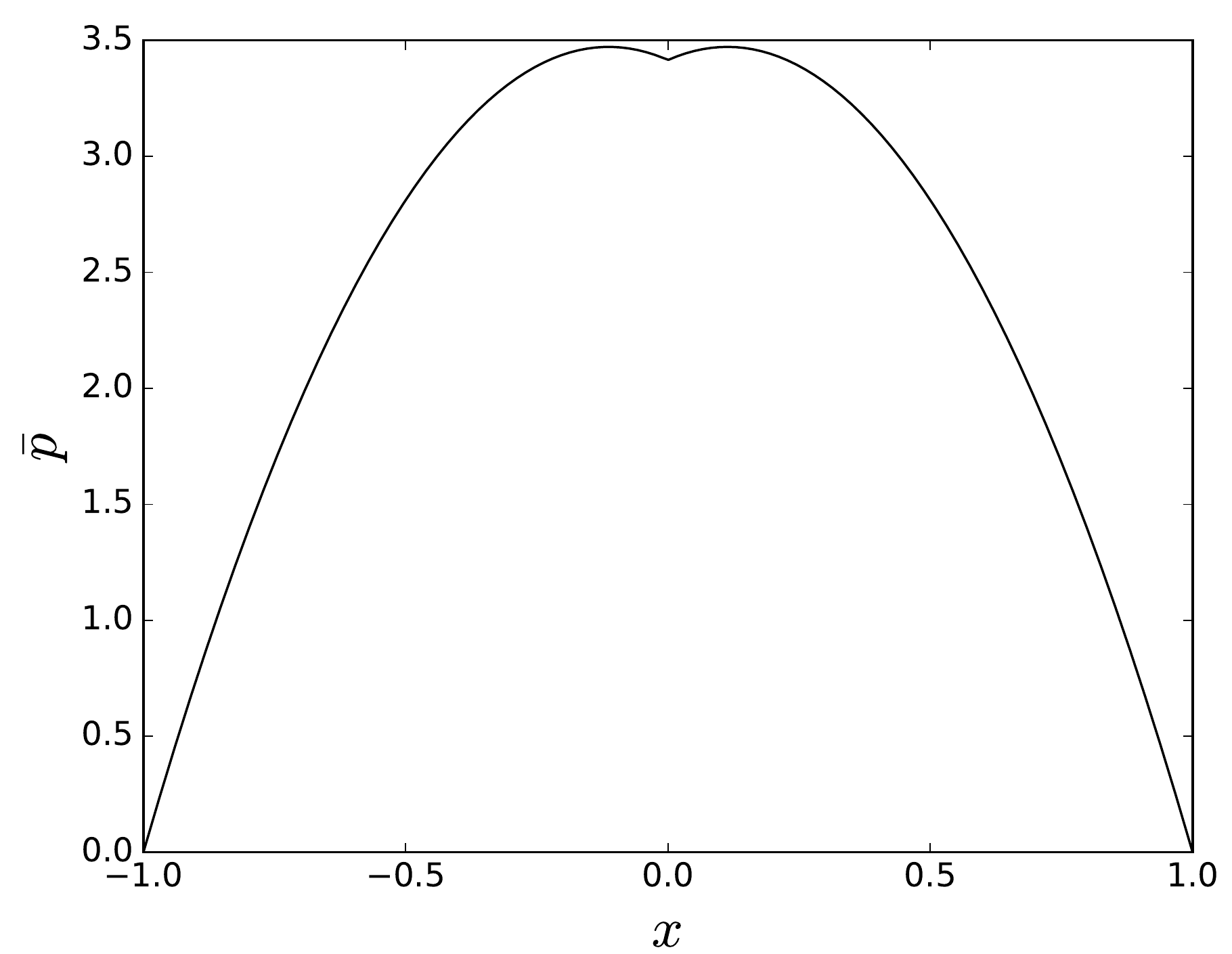}
   
    \end{subfigure}
    \
    \begin{subfigure}[b]{0.32\textwidth}
        \includegraphics[width=\textwidth]{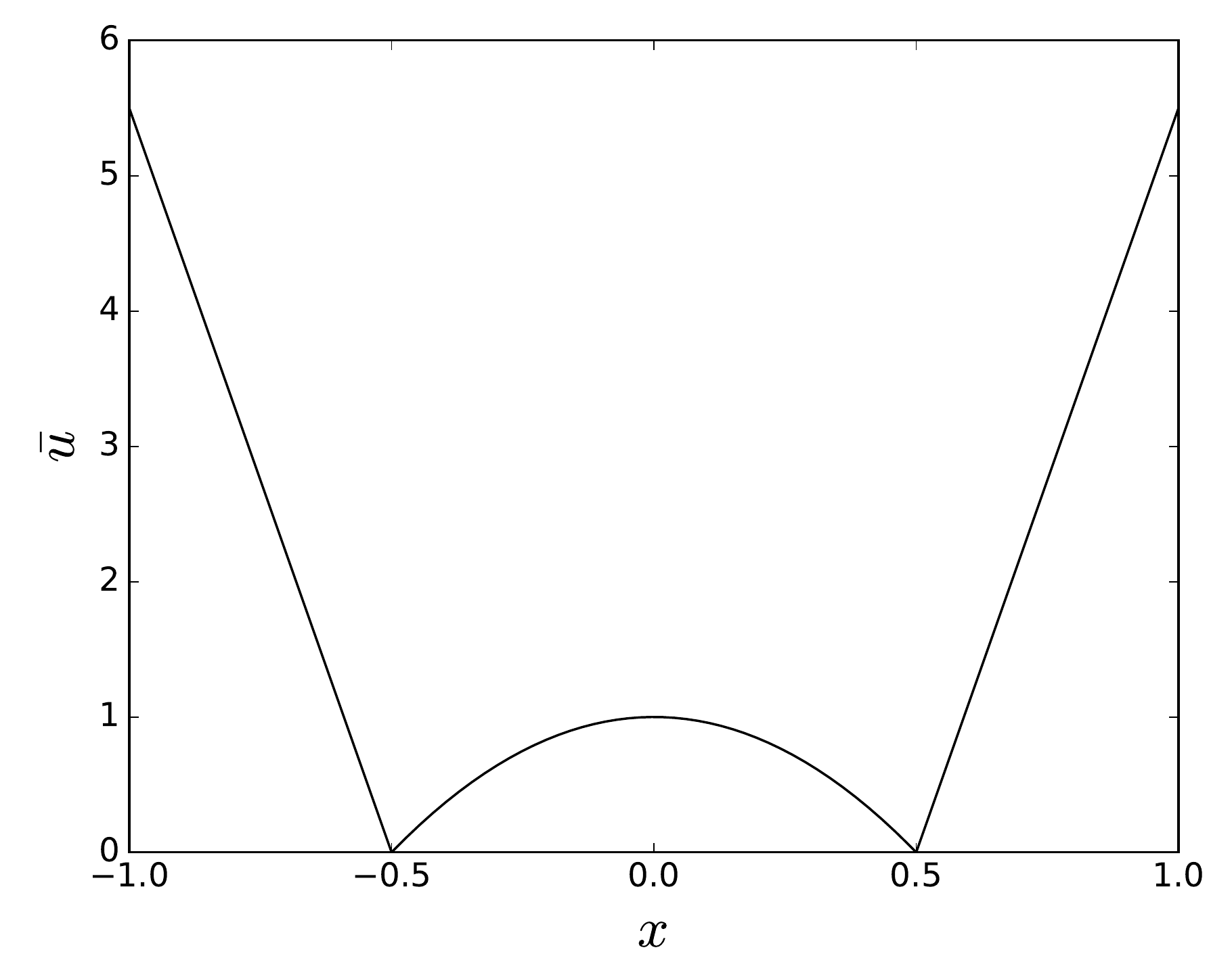}
       
    \end{subfigure}
    \caption{The state $\bar{y}$ (left), the adjoint state $\bar{p}$ (middle) and the control $\bar{u}$ (right)
     for the parameter setting in \eqref{eq:params}.}\label{fig:exact_sol}
\end{figure} 
According to \cite[Section~6.10]{Dobrowolski}, this implies that $\bar{p},\bar{u}\in H^{3/2-\varepsilon}(\Omega)$ 
for every $\varepsilon>0$ as indicated above. 
Compared to the case with a (smooth) elliptic equation, we thus obtain a substantial reduction of the regularity of the adjoint state. 
To be more precise, if one would consider the optimal control of Poisson's equations with our data (i.e., $y_d$ and $u_d$
as given in \eqref{eq:yd} and \eqref{eq:ud}, respectively), then 
the state and the adjoint would be elements of $H^2(\Omega)$.
Nevertheless, the regularity of the adjoint state in our example is still better than the minimum regularity 
guaranteed by the necessary optimality conditions in Theorem~\ref{theo:strong_stationarity}, which only 
yield $\bar p\in W_{\bar y}$, which is just $H^1_0(\Omega)$ here, since $\MM = \emptyset$.

\section{Numerical Results}\label{sec:numerics}

Let us report on the numerical results of the one-dimensional example constructed in the previous section. We choose the parameters 
\begin{equation}\label{eq:params}
m=1, \quad \alpha=\frac{11}{528},\quad \beta=\frac{1}{2}m(68\alpha +\sqrt{2}+4),\quad \nu=1
\end{equation}
so that the requirements in \eqref{eq:paramcond} are met with $\epsilon = 1$. 
The discrete optimal control problem is solved numerically by a regularization of the sub\-differential  $\partial\|\cdot\|_{L^1(\Omega)}$.  As before, we use the function
\begin{equation*}
\beta_\gamma(v)=\frac{2}{\pi}\arctan(\gamma v)
\end{equation*}
as regularization, i.e., we solve the regularized optimal control problem \eqref{prob_P_gamma}. Here $\gamma$ is the regularization and smoothing parameter. In the appendix we show that
\begin{equation*}
\bar{y}_\gamma\to \bar{y}\quad \text{in}\ H_0^1(\Omega),\qquad \bar{u}_\gamma \to \bar{u}\quad \text{in}\ L^2(\Omega)
\end{equation*} 
as $\gamma$ tends to infinity (cf. Theorem \ref{theo_conv_to_original_sol} and Theorem \ref{theo_strong_conv_P_gamma}). For further details concerning the regularization we refer to the appendix. This regularization leads to a smooth optimization problem and thus, the arising problem can be solved with Newton's method. Note that the regularity of the problem progressively decreases as $\gamma$ is increased. Consequently it is difficult to solve the problem numerically for large values of $\gamma$. Therefore, we first compute a solution of \eqref{prob_P_gamma} for $\gamma=1$. The initial values for Newton's method are set to $y_0=0$, $p_0=0$ and $u_0=1$. Afterwards we increase $\gamma$ by the rule $\gamma=1.5\cdot \gamma$ and solve again problem \eqref{prob_P_gamma}. As initial values for Newton's method we use the previously computed solution. This process is repeated until $\gamma=10^{-20}$. 

 Figure \ref{fig:approx_sol} illustrates the numerical solution for $h=1/100$ and $\gamma=10^{-20}$. We point out that even the kink of the adjoint state is visible in the approximation. Moreover, if we look closely at $\bar{p}_h$ we even see the effect of the regularization.

\begin{figure}[h]
\centering
\begin{subfigure}[b]{0.32\textwidth}
        \includegraphics[width=\textwidth]{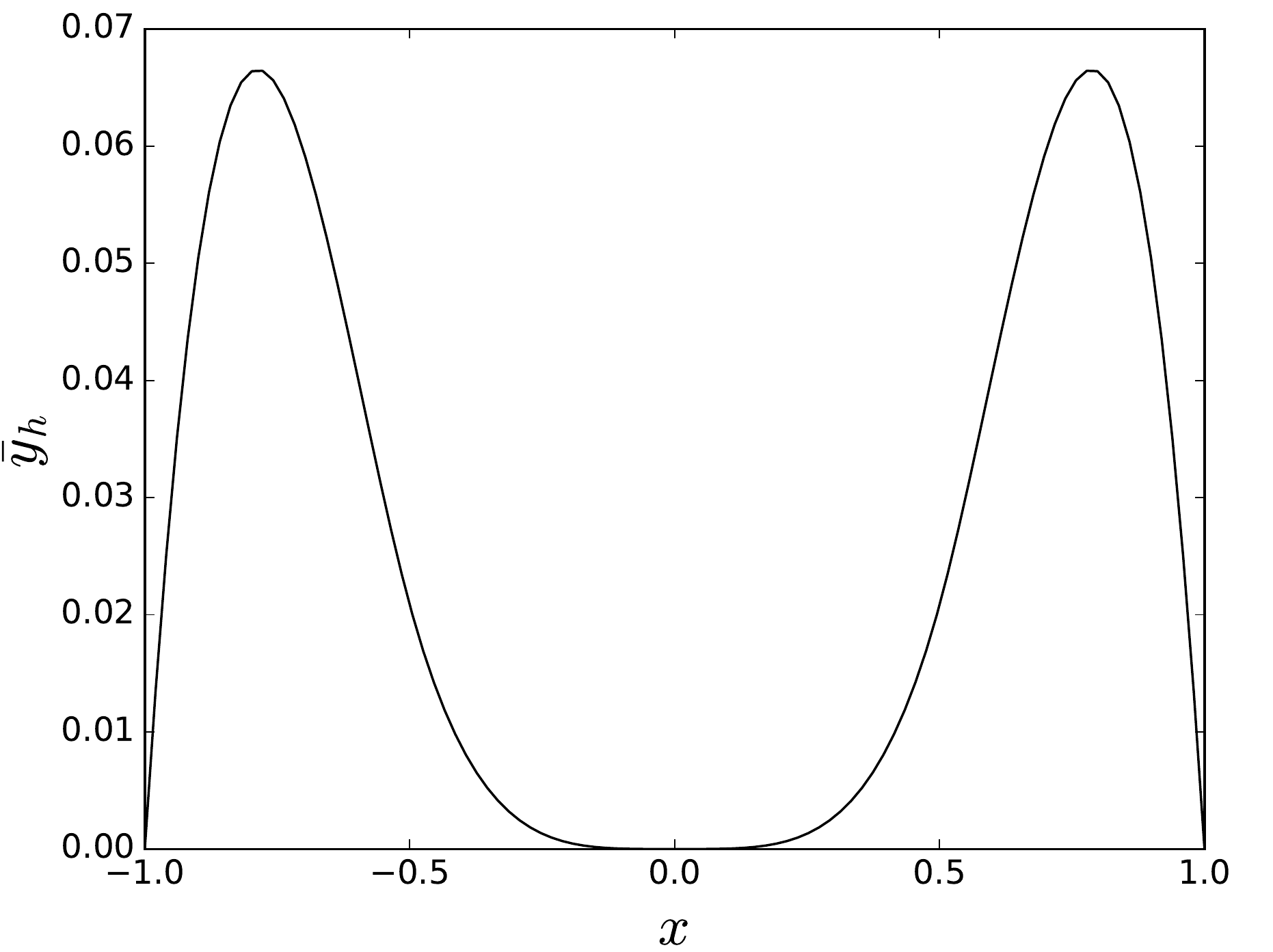}
       
    \end{subfigure}
    \
    \begin{subfigure}[b]{0.32\textwidth}
        \includegraphics[width=\textwidth]{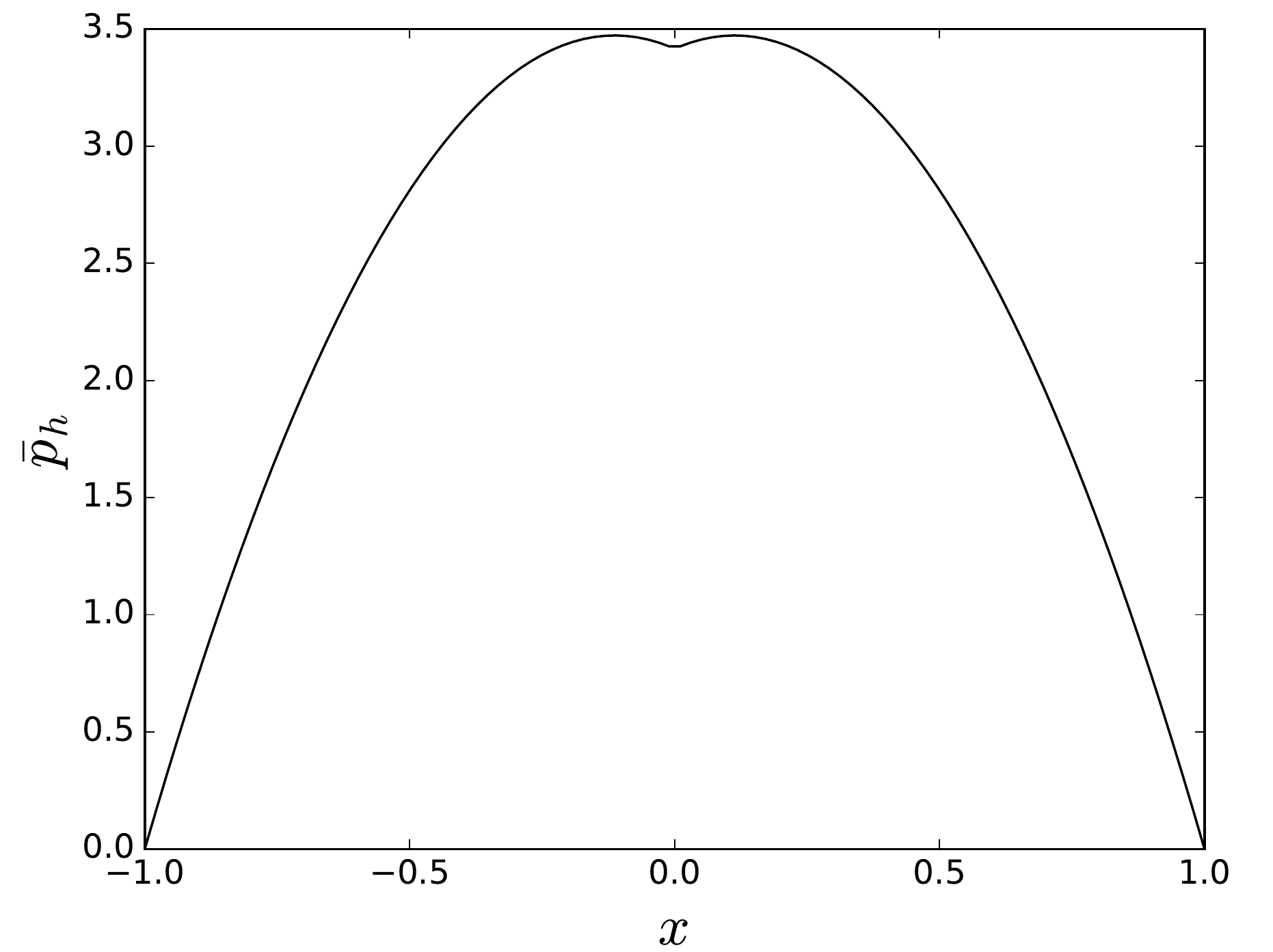}
   
    \end{subfigure}
    \
    \begin{subfigure}[b]{0.32\textwidth}
        \includegraphics[width=\textwidth]{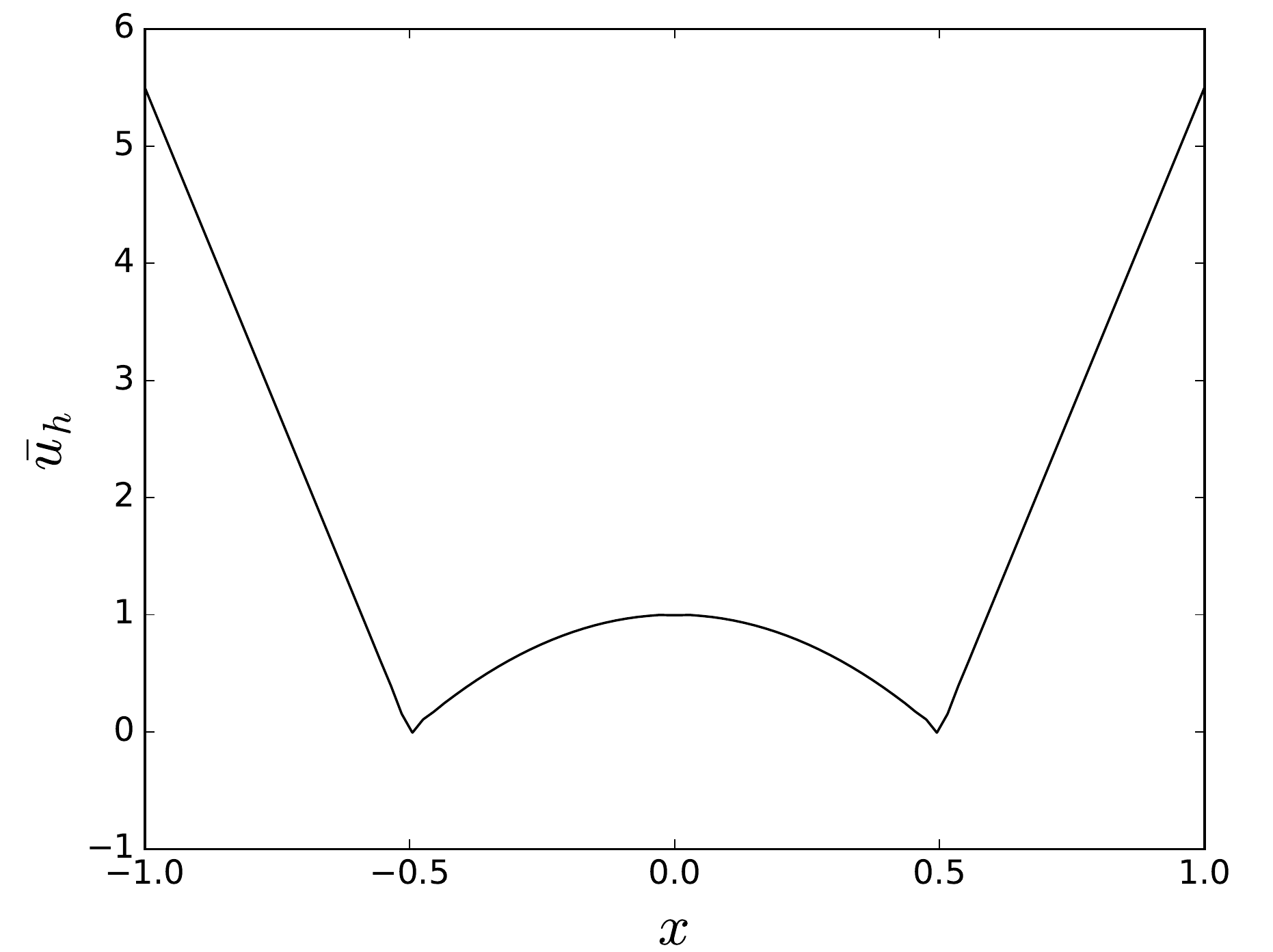}
       
    \end{subfigure}
    \caption{The state $\bar{y}_h$ (left), the adjoint state $\bar{p}_h$ (middle) and the control $\bar{u}_h$ (right) for $h=\frac{1}{100}$.}\label{fig:approx_sol}
\end{figure} 

The overall error consists of two contributions, namely the regularization error and the discretization error. We observe in Figure \ref{fig:saturation} that for a fixed mesh size $h=1/n$ the discretization error dominates provided that $\gamma$ is sufficiently large.
\begin{figure}[h]
\centering
\begin{subfigure}[b]{0.32\textwidth}
        \includegraphics[width=\textwidth]{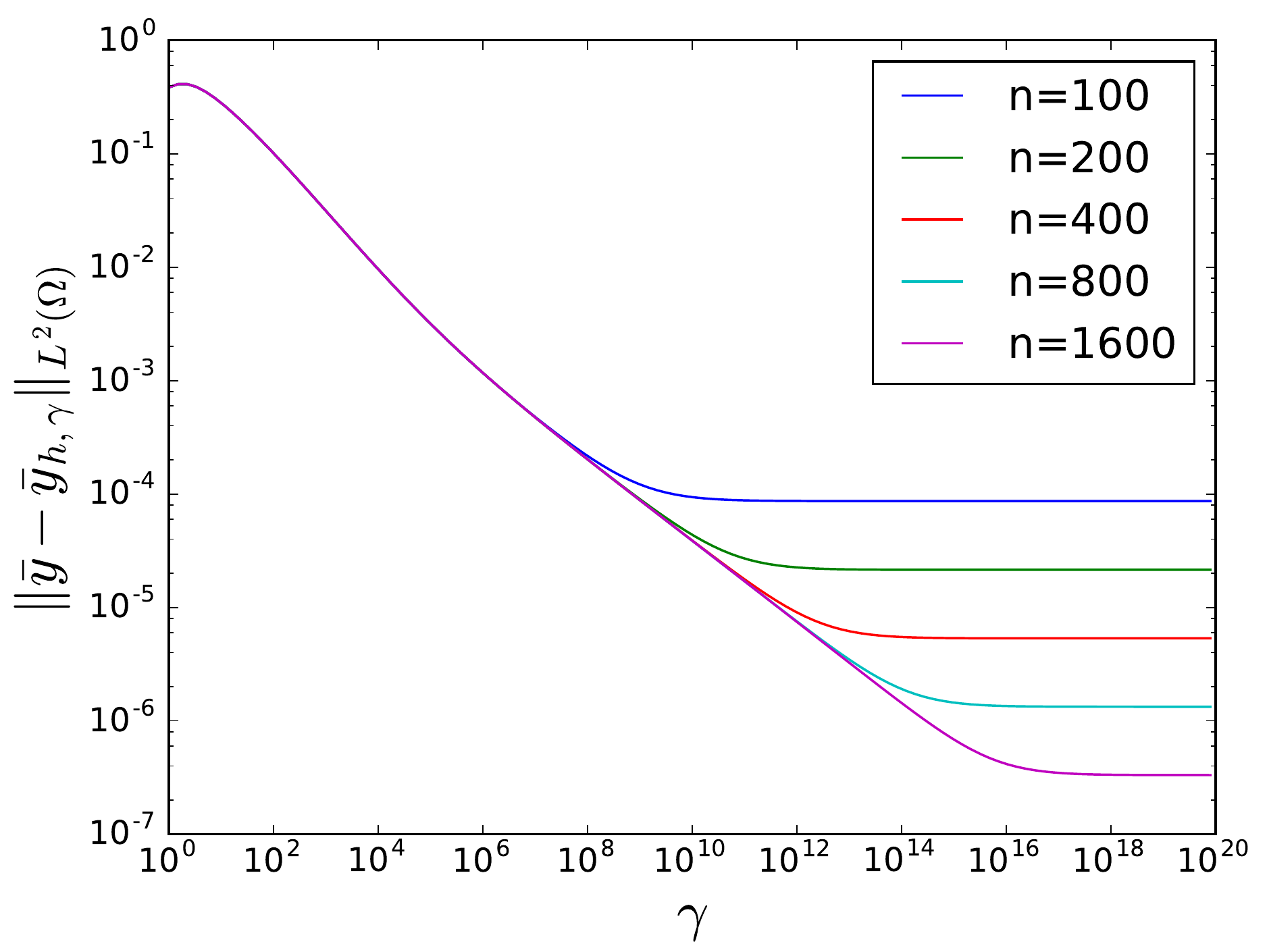}
       
    \end{subfigure}
    \
    \begin{subfigure}[b]{0.32\textwidth}
        \includegraphics[width=\textwidth]{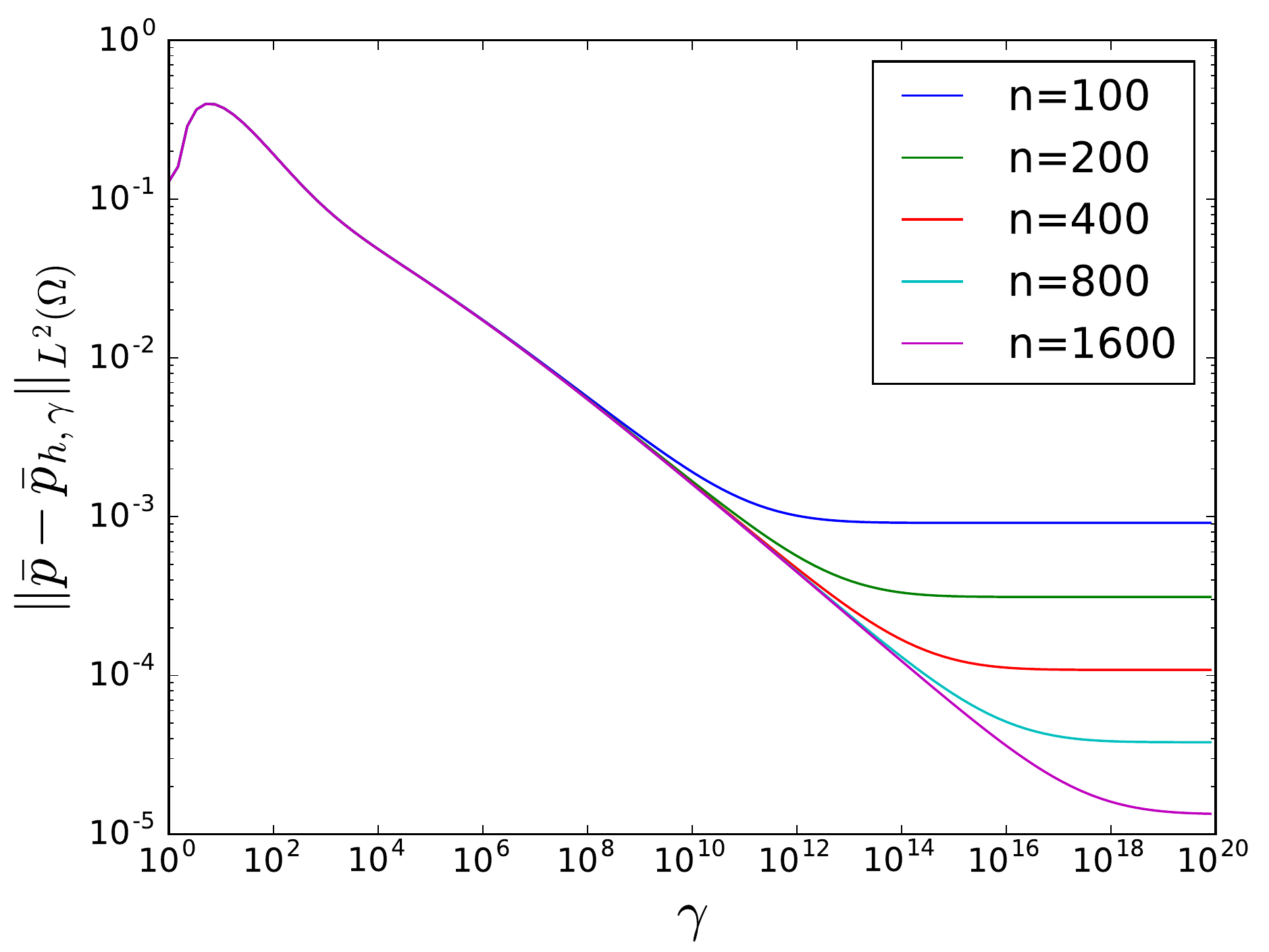}
   
    \end{subfigure}
    \
    \begin{subfigure}[b]{0.32\textwidth}
        \includegraphics[width=\textwidth]{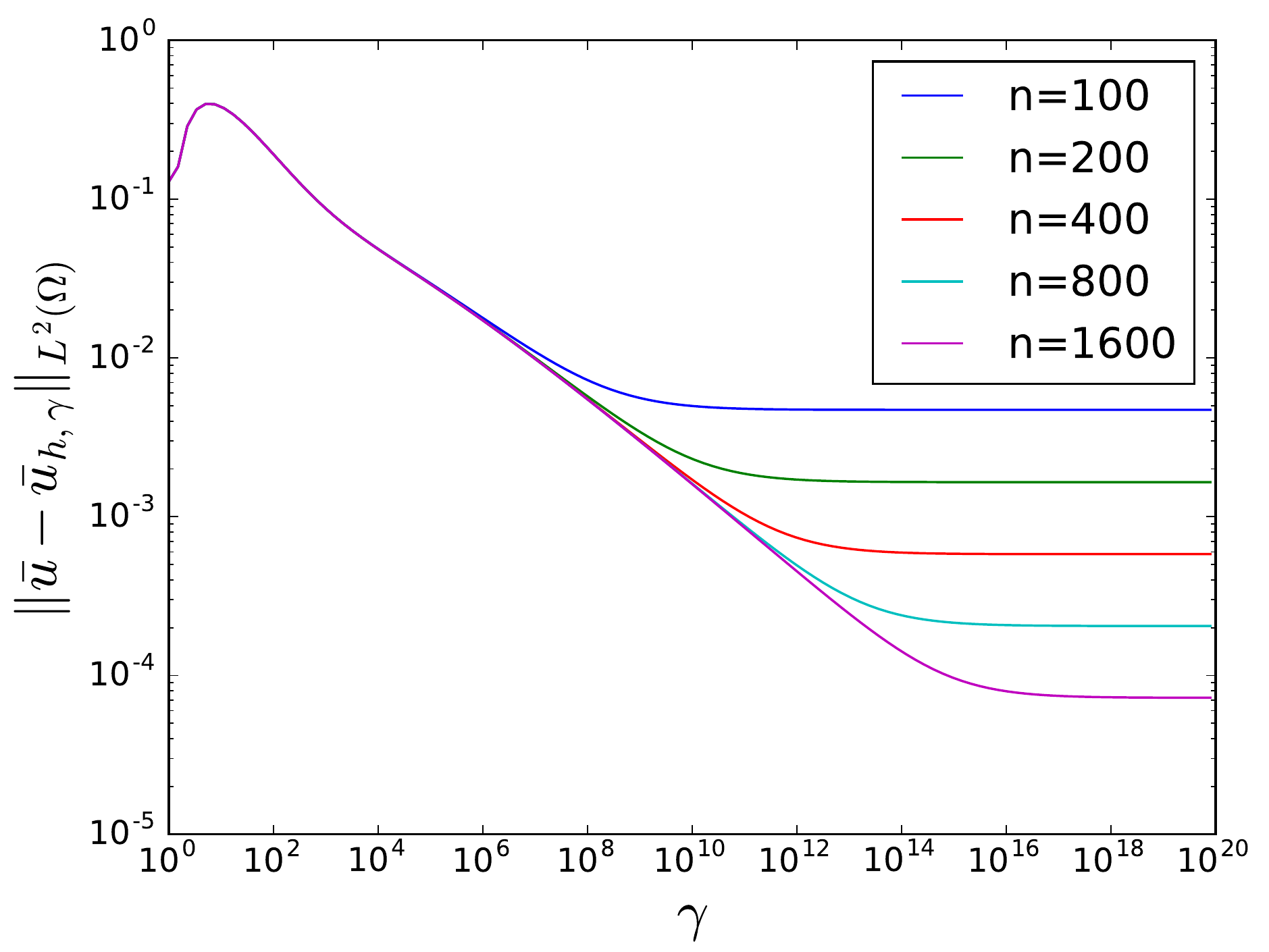}
       
    \end{subfigure}
    \caption{Regularization error for the state (left), the adjoint state (middle) and the control (right) on different meshes.}\label{fig:saturation}
\end{figure} 

The $L^2$-errors for this example are depicted in Figure \ref{fig:l2-error} for the choice $\gamma=10^{-20}$.

\begin{figure}[h]
\centering
\begin{subfigure}[b]{0.32\textwidth}
        \includegraphics[width=\textwidth]{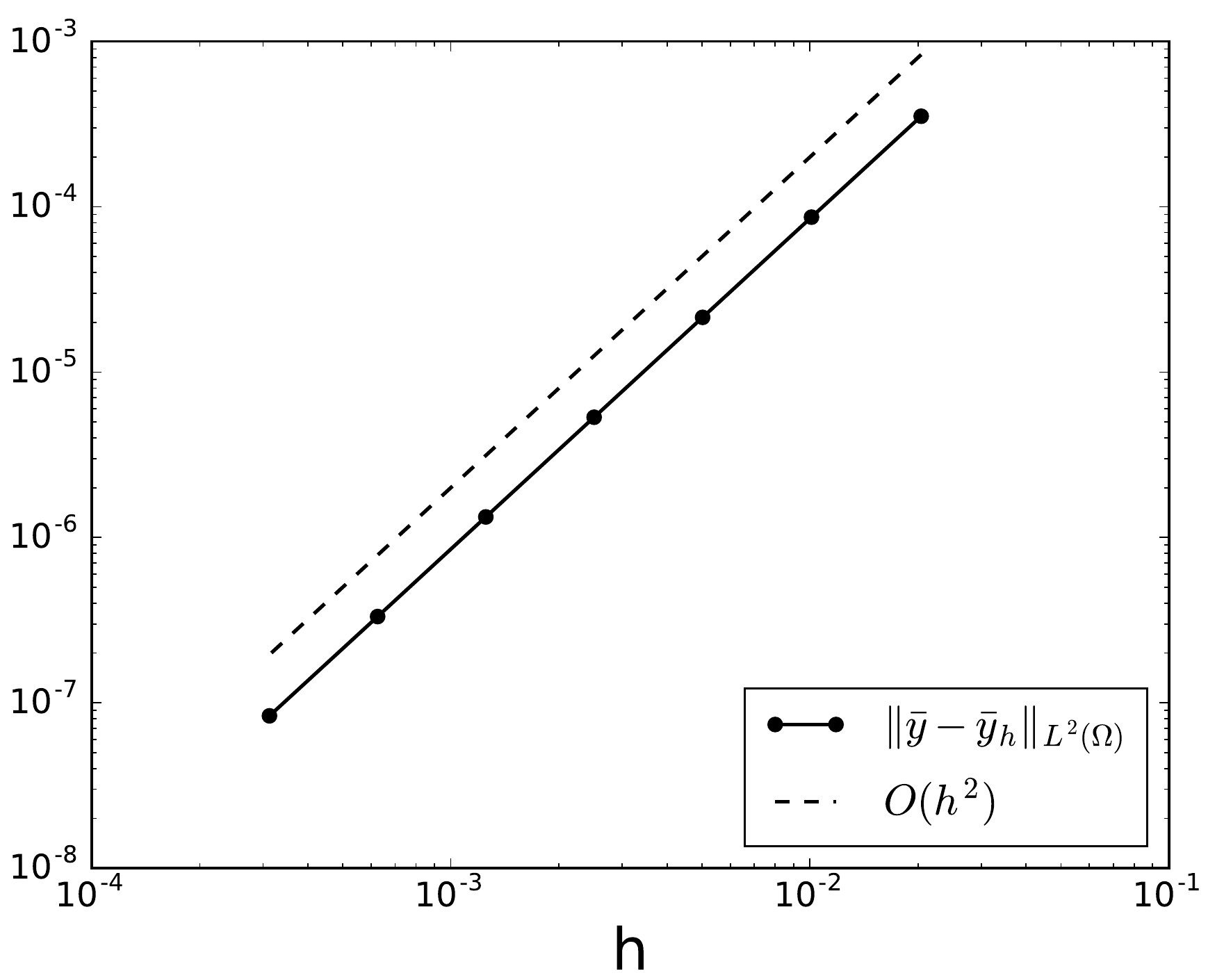}
       
    \end{subfigure}
    \
    \begin{subfigure}[b]{0.32\textwidth}
        \includegraphics[width=\textwidth]{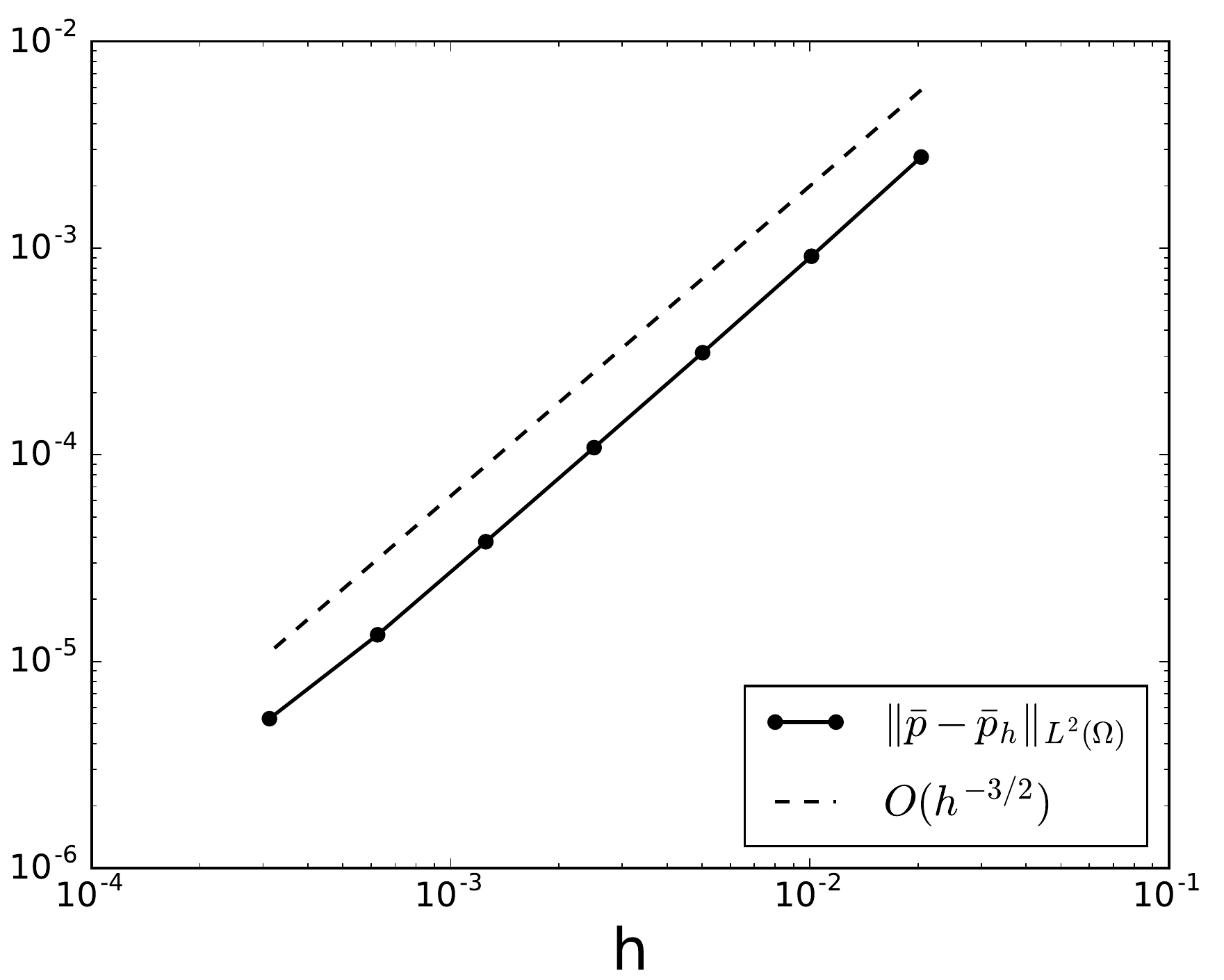}
   
    \end{subfigure}
    \
    \begin{subfigure}[b]{0.32\textwidth}
        \includegraphics[width=\textwidth]{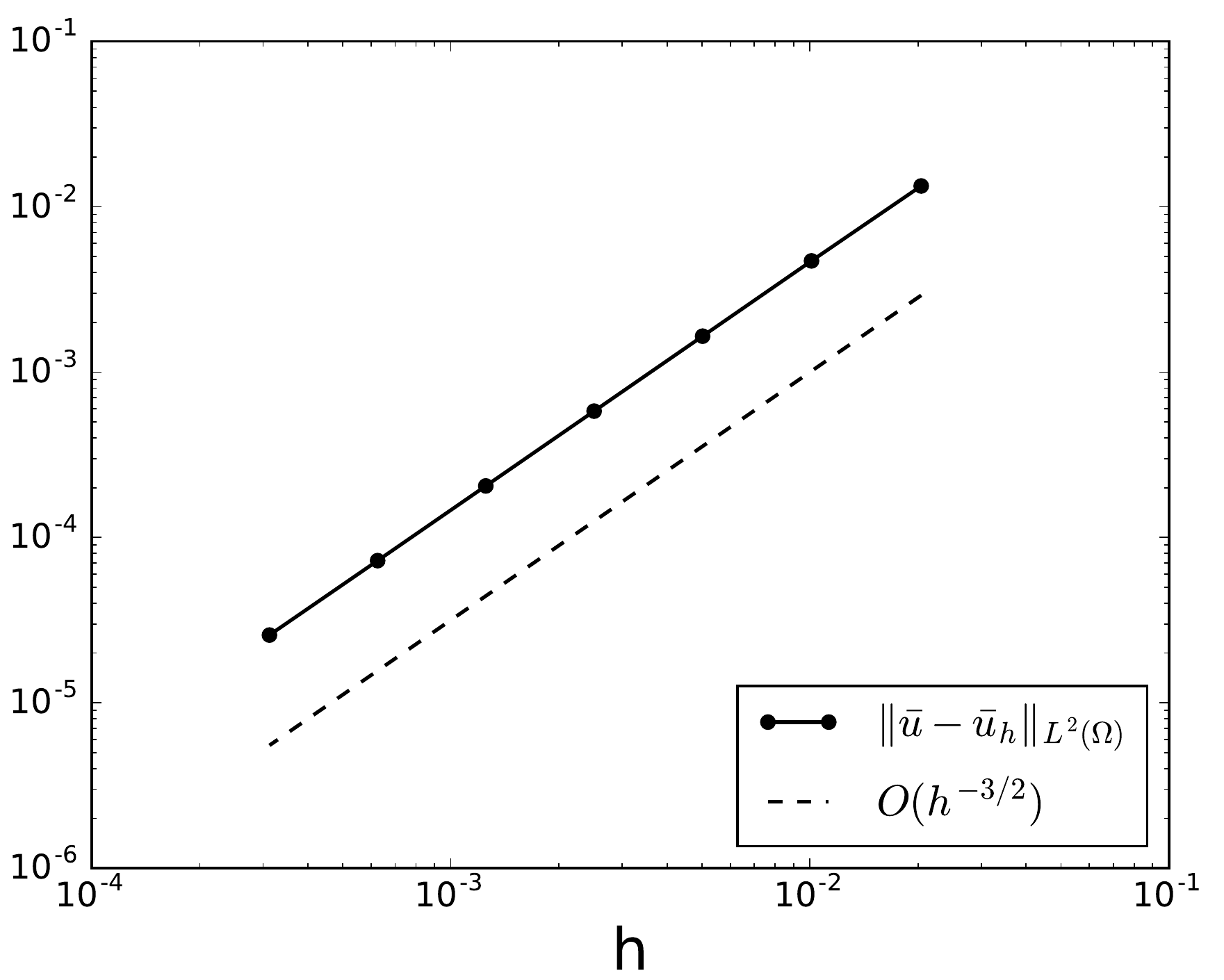}
       
    \end{subfigure}
    \caption{$L^2$-error for the state (left), for the adjoint state (middle) and for the control (right).}\label{fig:l2-error}
\end{figure} 

\begin{table}
  \centering
\begin{tabular}{p{1.4cm} p{1.8cm} p{1.4cm} p{1.8cm} p{1.4cm} p{1.8cm} p{1.4cm}}
$h$ & $e(\bar{y},h)$ & $EOC_{\bar{y}}$ & $e(\bar{p},h)$ & $EOC_{\bar{p}}$ & $e(\bar{u},h)$ & $EOC_{\bar{u}}$\\
\hline
1/50 & 3.5320e-04 & \quad- & 2.7605e-03 & \quad- & 1.3373e-02 & \quad -\\
1/100 & 8.6626e-05 & 1.9983 & 9.1383e-04  & 1.5719 & 4.7088e-03 & 1.4841\\
1/200 & 2.1401e-05 & 1.9993 & 3.1204e-04 & 1.5390 & 1.6505e-03 & 1.5016 \\
1/400 & 5.3368e-06 & 1.9997 & 1.0836e-04 & 1.5203 & 5.8102e-04 &  1.5008\\
1/800 & 1.3310e-06 & 1.9999 & 3.7974e-05  & 1.5101 & 2.0498e-04 & 1.5004\\
1/1600 & 3.3239e-07 & 1.9997 & 1.3459e-05 & 1.4951 & 7.2410e-05 & 1.4999\\
\hline
\end{tabular}
\caption{Experimental order of convergence.}\label{table:EOC}
\end{table}

In order to verify our theoretical convergence rate, we compute the experimental order of convergence which for the control is given by
\begin{equation*}
EOC_{\bar{u}}:=\frac{\log(e(\bar{u},h_1))-\log(e(\bar{u},h_2))}{\log(h_1)-\log(h_2)}, 
\end{equation*}
where $h_1$ and $h_2$ denote two consecutive mesh sizes and $e(\bar{u},h)$ is the error in the $L^2$-norm, i.e.
\begin{equation*}
e(\bar{u},h):=\|\bar{u}-\bar{u}_h\|_{L^2(\Omega)}.
\end{equation*} 
$EOC_{\bar{y}}$ and $EOC_{\bar{p}}$ are defined analogously.
Table \ref{table:EOC} shows the experimental order of convergence for this example. One can see that $EOC_{\bar{y}}$ is approximately two which is in agreement with our theoretical result in Theorem \ref{theo_Nochetto2}. Moreover, we observe that $EOC_{\bar{p}}$ and $EOC_{\bar{u}}$ are approximately $3/2$ which is better than predicted by our theory 
(cf.\ Theorem~\ref{conv_rates}). 
The experimental rate of $3/2$ is however in line with the regularity of the adjoint state in this example.
In order to see the predicted convergence rate of Theorem \ref{conv_rates} 
we would have to construct a test example such that $\bar{p}, \bar{u}\in H^1(\Omega)$, 
but $\bar{p}, \bar{u}\notin H^{1+\epsilon}(\Omega)$ for every $\epsilon>0$. 
In \cite{MeyerThoma2013}, this was done for the optimal control of the obstacle problem, which is significantly easier to handle 
due to the simpler structure of the second-order sufficient conditions in this case. 
We expect that the construction of an example with minimal regularity requires to turn to higher spatial dimensions 
which gives rise to future research.

\begin{appendices}
\appendix
\renewcommand{\thesection}{Appendix \Alph{section}}
\renewcommand{\thesubsection}{\Alph{section}.\arabic{subsection}}
\renewcommand{\theequation}{\Alph{section}.\arabic{equation}}
\renewcommand{\thetheorem}{\Alph{section}.\arabic{theorem}}

\section{$\boldsymbol{L^\infty}$-Error Estimates for the State}\label{sec:Appendix}
This section is devoted to the derivation of the $L^\infty$-error estimate for the variational inequality \eqref{VI} used in the error analysis of Section \ref{sec:error_analysis}. The proof is an adaptation of the technique introduced in \cite{Nochetto1988} and is based on a regularization of the variational inequality. The main difference from the analysis presented in \cite{Nochetto1988} is that we have the generic regularity $u\in H^1(\Omega)$, whereas in \cite{Nochetto1988} $u\in L^\infty(\Omega)$ is assumed.\\
We will first establish $L^\infty$-error estimates for the regularized problem as well as for its discretization. Based on these results we finally prove an error estimate for the original problem.

\subsection{Error Estimates for the Regularized Variational Inequality}\label{sec_reg_VI}
We now explain a regularization procedure for \eqref{VI}. Note first that problem \eqref{VI} can be rewritten as
\begin{align}\label{eq:pdey}
y\in H_0^1(\Omega),\quad a(y,v)+(\beta(y),v)=\langle u,v\rangle\quad \forall v\in H_0^1(\Omega)
\end{align}

with $\beta(y)\in\partial\|\cdot\|_{L^1(\Omega)}(y)$, where $\partial\|\cdot\|_{L^1(\Omega)}$ denotes the subdifferential of $H_0^1(\Omega)\ni y\mapsto \|y\|_{L^1(\Omega)}\in\R$. For the regularization of \eqref{VI} we substitute $\beta$ by a bounded and globally smooth function $\beta_\gamma$. We use the function
\begin{equation*}
\beta_\gamma(s):=\frac{2}{\pi}\arctan(\gamma s).
\end{equation*}

The regularized problem of \eqref{VI} is then given by
\begin{equation}\label{eq:pdegamma}
y_\gamma\in H_0^1(\Omega),\quad a(y_\gamma,v)+(\beta_\gamma(y_\gamma),v)=\langle u,v\rangle\quad \forall v\in H_0^1(\Omega).
\end{equation}

In the following we collect some properties of the regularized problem. We start with a regularity result.
\begin{lemma}\label{lemma_reg_VI_regularity}
For every $u\in H^{1}(\Omega)$ the solution $y_\gamma$ of the regularized problem \eqref{eq:pdegamma} satisfies $y_\gamma\in W^{1,\infty}(\Omega)$.
\end{lemma}

\begin{proof}
We apply a boot strapping argument. By construction $\beta_\gamma\in L^\infty(\Omega)$ with a norm independent of $\gamma$. Moreover, we have the continuous embeddings $u\in H^1(\Omega)\hookrightarrow L^p(\Omega)$ for all $p<\infty$, if $d\leq 2$, and $u\in H^1(\Omega)\hookrightarrow L^p(\Omega)$ for all $p\leq 6$, if $d=3$. Since our domain $\Omega$ is assumed to be $W^{2,p}$-regular, it follows via boot strapping that
\begin{equation*}
\|y_\gamma\|_{W^{2,p}(\Omega)}\leq C\left( \|\beta_\gamma(y_\gamma)\|_{L^\infty(\Omega)}+\|u\|_{L^p(\Omega)}\right)\leq C\left(1+\|u\|_{H^1(\Omega)}\right),
\end{equation*}
which implies the desired regularity of $y_\gamma$.
\end{proof}

The following result shows that the solution of the regularized problem converges to the solution of the original problem.

\begin{theorem}\label{theo_conv_to_original_sol}
Let $\{u_\gamma \} \subset H^{-1}(\Omega)$ be a sequence such that $u_\gamma \to u$ in $H^{-1}(\Omega)$
as $\gamma \to \infty$. Denote the solutions of \eqref{eq:pdegamma} associated with $u_\gamma$ by $y_\gamma$, 
while $y$ is the solution of the VI in \eqref{eq:pdey}. 
Then $y_\gamma\to y$ in $H_0^1(\Omega)$ as $\gamma\to \infty$.
\end{theorem}

\begin{proof}
Equality \eqref{eq:pdegamma}
 is equivalent to the variational inequality
\begin{equation}\label{reg_var_eq}
a(y_\gamma, v-y_\gamma) + \int_\Omega j_\gamma(v)\ \mathrm dx - \int_\Omega j_\gamma(y_\gamma)\ \mathrm dx \geq 
\dual{u_\gamma}{v - y_\gamma}
\quad \forall v\in H_0^1(\Omega)
\end{equation}
with
\begin{equation*}
j_\gamma(v)= \frac{2}{\pi}\left[v\arctan(\gamma v) -\frac{1}{2\gamma} \ln\left(\frac{1}{\gamma^2}+v^2\right)\right]
\end{equation*}
(note that $j_\gamma'(v)= \frac{2}{\pi}\arctan(\gamma v)$).

If we choose $v=y_\gamma$ in \eqref{eq:pdegamma}, 
then the monotonicity of $\arctan$ implies
\begin{equation*}
\|y_\gamma\|_{H^1(\Omega)}^2\leq C\|u_\gamma\|_{H^{-1}(\Omega)}
\end{equation*}
with a constant $C$ independent of $\gamma$. Hence, we have the weak convergence
\begin{equation}\label{weak_H1_conv}
y_\gamma \rightharpoonup  y\quad \text{in} \ H_0^1(\Omega).
\end{equation}
As we will prove next, the weak limit indeed satisfies the VI in \eqref{eq:pdey}, which justifies the 
above notation.
Due to $\lim_{x\to \pm\infty}\arctan(x) =\pm\frac{\pi}{2}$ and 
$\lim_{\gamma\to \infty}\frac{1}{\pi\gamma}\ln(\gamma^{-2}+v^2)=0$ we have
for every $v\in H^1_0(\Omega)$ that
$\lim_{\gamma\to\infty}j_\gamma(v) =|v|$ a.e.\ in $\Omega$. Since it further holds
$|j_\gamma(v)| \leq C(|v|+1)$, Lebesgue's dominated convergence theorem implies
\begin{equation}\label{j_gamma_v}
\lim_{\gamma\to\infty}j_\gamma(v) =|v|\quad \text{in}\ L^1(\Omega).
\end{equation}
Moreover, due to the weak convergence of $y_\gamma$ in $H_0^1(\Omega)$, 
there exists a subsequence again denoted by $y_\gamma$ such that $y_\gamma \to y$ in $L^2(\Omega)$.
Thus, we obtain (possibly after passing over to a subsequence) 
$\lim_{\gamma\to\infty}j_\gamma(y_\gamma) = |y|$ a.e.\ in $\Omega$ and 
\begin{equation*}
\|j_\gamma(y_\gamma)\|_{L^2(\Omega)} \leq C (\|y_\gamma\|_{L^2(\Omega)}+1)\leq C.
\end{equation*}
Hence, a modified version of Lebesgue's dominated convergence theorem, see \cite[Lemma~A.2]{BetzMeyer2015}, 
implies that
\begin{equation}\label{j_gamma_y_gamma}
\lim_{\gamma\to\infty} j_\gamma(y_\gamma) = |y|\quad\text{in}\ L^1(\Omega).
\end{equation}
Now, we have everything at hand to pass to the limit in \eqref{eq:pdegamma}.
By \eqref{weak_H1_conv}, \eqref{reg_var_eq}, \eqref{j_gamma_v}, \eqref{j_gamma_y_gamma}, 
and the strong convergence of $\{u_\gamma\}$ 
it follows for every $v\in H^1_0(\Omega)$ that
\begin{align}\label{liminf}
\int_\Omega |\nabla y|^2 \ \mathrm dx &\leq \liminf_{\gamma\to \infty}\int_\Omega |\nabla y_\gamma|^2\ \mathrm dx\leq \limsup_{\gamma\to \infty}\int_\Omega |\nabla y_\gamma|^2\ \mathrm dx\nonumber\\
&\leq \limsup_{\gamma\to \infty}\left[\int_\Omega \nabla y_\gamma \cdot\nabla v \ \mathrm dx + \int_\Omega j_\gamma (v)\ \mathrm dx -\int_\Omega j_\gamma(y_\gamma)\ \mathrm dx - \dual{u_\gamma}{v-y_\gamma} \right]\nonumber\\
 &= \int_\Omega \nabla y\cdot \nabla v\ \mathrm dx +\int_\Omega |v| \ \mathrm dx - \int_\Omega |y| \ \mathrm dx 
 - \dual{u}{v-y}.
\end{align}
Thus, $y$ solves the variational inequality and consequently the weak limit $y$ in \eqref{weak_H1_conv} is unique, 
which implies the weak convergence of the whole sequence.

Finally, 
to see the strong convergence of $\{y_\gamma\}$, 
we choose $v=y$ in \eqref{liminf} in order to obtain
$\lim_{\gamma\to\infty}\|y_\gamma\|_{H^1(\Omega)}=\|y\|_{H^1(\Omega)}$, which, 
together with the weak convergence,
implies the desired strong convergence.
\end{proof}

\begin{theorem}\label{theo_reg_error}
The regularization error can be bounded by
\begin{align}
&\|y-y_\gamma\|_{L^2(\Omega)}\leq C\gamma^{-1/3}\label{eq:L2est}\\
&\|y-y_\gamma\|_{L^\infty(\Omega)}\leq C\gamma^{-2/(3(2+s))}\label{eq:Linfest}
\end{align}
with $s>\max\{d/2, 1\}$.
\end{theorem}

\begin{proof}
Subtracting \eqref{eq:pdegamma} from \eqref{eq:pdey} yields
\begin{equation}\label{eq:pdediff}
a(e_\gamma, v)=(\beta_\gamma(y_\gamma)-\beta(y),v)\quad \forall v\in H_0^1(\Omega)
\end{equation}
with $e_\gamma:= y-y_\gamma$.
We already know that $y, y_\gamma \in W^{1,\infty}(\Omega)$, so we may test \eqref{eq:pdediff} with $e^{2p +1}$, 
where $p\in \N \cup\{0\}$ is fixed but arbitrary.
Then we obtain completely analogously to the proof of \cite[Theorem~2.1]{Nochetto1988}
\begin{equation}\label{eq:Nochetto}
   C\, \frac{2p+1}{(p+1)^2} \|e_\gamma\|_{L^{2(p+1)}(\Omega)}^{2(p+1)} \leq (\nabla e_\gamma, \nabla(e_\gamma^{2p+1}))
   = (\beta_\gamma(y_\gamma) - q, e_\gamma^{2p+1}).
\end{equation}
We split the right-hand side in three parts, i.e.
\begin{align*}
     & (\beta_\gamma(y_\gamma) - q, e_\gamma^{2p+1}) \\
     & = \int_{\Omega_\gamma^0} \big(\beta_\gamma(y_\gamma) - q\big) e_\gamma^{2p+1}\,\mathrm d x
     + \int_{\Omega_\gamma^1} \big(\beta_\gamma(y_\gamma) - q\big) e_\gamma^{2p+1}\,\mathrm d x
     + \int_{\Omega_\gamma^2} \big(\beta_\gamma(y_\gamma) - q\big) e_\gamma^{2p+1}\,\mathrm d x \\
    & =: A^0_\gamma +  A^1_\gamma + A^2_\gamma,
\end{align*}
where 
\begin{equation*}
\begin{aligned}
    \Omega_\gamma^0 &:= \{x\in \Omega : |y(x)|, |y_\gamma(x)| \leq \gamma^{-1/a} \}\\
    \Omega_\gamma^1 &:= \{x\in \Omega \setminus \Omega_\gamma^0 : |y(x)| > \gamma^{-1/a}  \}, \quad
    \Omega_\gamma^2 := \Omega \setminus (\Omega_\gamma^0 \cup \Omega_\gamma^1)
\end{aligned}
\end{equation*}
with some $a>1$ to be specified later. On $\Omega_\gamma^0$ we find
\begin{equation*}
    A_\gamma^0 \leq 
    \|\beta_\gamma(y_\gamma) - q\|_{L^\infty(\Omega)} \int_{\Omega_\gamma^0} (2\gamma^{-1/a})^{(2p+1)}\, \mathrm d x
    \leq C \,(2\gamma^{-1/a})^{(2p+1)},
\end{equation*}
where we used that $-1\leq \beta_\gamma(y_\gamma), q \leq 1$ a.e.\ in $\Omega$.
Now, if $\gamma \geq 1$ (which is no restriction, since $\gamma$ tends to $\infty$ anyway), then,
for every $r\in \R\setminus [-\gamma^{-1/a} , \gamma^{-1/a} ]$, one obtains
\begin{equation}\label{eq:atanest}
\begin{aligned}
    |\beta_\gamma(r) - \sgn(r)| 
    &= |\sgn(\gamma r) - \tfrac{2}{\pi} \arctan(\tfrac{1}{\gamma r}) - \sgn(r)| \\
    & \leq \tfrac{2}{\pi} \arctan(\gamma^{-(1-1/a)}) \leq \gamma^{-(1-1/a)}.
\end{aligned}
\end{equation}

On $\Omega_\gamma^1$ we have
\begin{equation*}
    |y(x)| > \gamma^{-1/a}\quad \Longrightarrow \quad q(x) = \sgn(y(x)).
\end{equation*}

Thus, the monotonicity of $\beta_\gamma$ and \eqref{eq:atanest} yield
\begin{equation*}
\begin{aligned}
    A_\gamma^1 &= \int_{\Omega_\gamma^1} \big(\beta_\gamma(y_\gamma) - \beta_\gamma(y)\big) (y - y_\gamma)^{2p+1}\,\mathrm d x
    + \int_{\Omega_\gamma^1} \big(\beta_\gamma(y) - \sgn(y)\big) e_\gamma^{2p+1}\,\mathrm d x\\
   & \leq \gamma^{-(1-1/a)}\,\|e_\gamma\|_{L^{2p+1}(\Omega)}^{2p+1}. 
\end{aligned}
\end{equation*}
On $\Omega_\gamma^2$ we have $|y_\gamma| > \gamma^{-1/a}$ and thus, we can argue similarly to arrive at
\begin{equation*}
\begin{aligned}
    A_\gamma^2 &= \int_{\Omega_\gamma^2} \big(\beta_\gamma(y_\gamma) - \sgn(y_\gamma)\big) e_\gamma^{2p+1}\,\mathrm d x
    + \int_{\Omega_\gamma^1} \big(\sgn(y_\gamma) - q\big)(y - y_\gamma)^{2p+1}\,\mathrm d x\\
    & \leq \gamma^{-(1-1/a)}\,\|e_\gamma\|_{L^{2p+1}(\Omega)}^{2p+1},
\end{aligned}
\end{equation*}
where we used that the convex subdifferential of $\R \ni r \mapsto |r| \in \R$ is maximal monotone.

Altogether we obtain 
\begin{equation*}
  \|e_\gamma\|_{L^{2(p+1)}(\Omega)}^{2(p+1)} 
     \quad \leq C \,\frac{(p+1)^2}{2p + 1} \, 2^{2p + 1}\, \big(\gamma^{-1/a}\big)^{2p+1}
    + C \,\frac{(p+1)^2}{2p + 1} \, \gamma^{-(1-1/a)}\, \|e_\gamma\|_{L^{2p+1}(\Omega)}^{2p+1}.
\end{equation*}

For $p=0$ we obtain
\begin{equation*}
    \|e_\gamma\|_{L^{2}(\Omega)}^{2}
    \leq C \, \gamma^{-1/a} + C \,\gamma^{-(1-1/a)}\,\|e_\gamma\|_{L^1(\Omega)}. 
\end{equation*}
Thus, Young's and H\"older's inequality lead to
\begin{equation*}
    \|e_\gamma\|_{L^{2}(\Omega)} \leq C\, \max\{\gamma^{-1/(2a)}, \gamma^{-(1-1/a)}\}.
\end{equation*}
Now we can specify $a$. In order to equilibrate both error contributions we fix $a$ such that 
$1/(2a) = 1-1/a$, i.e. $a = 3/2>1$, and consequently \eqref{eq:L2est} is proved.

For the proof of \eqref{eq:Linfest} let $\alpha > 1$ be given and define
\begin{equation*}
\begin{alignedat}{3}
    \Omega_\gamma^{++} &:= \{x\in \Omega : y(x) > \gamma^{-1/\alpha} \}, & \quad
    \Omega_\gamma^{--} &:= \{x\in \Omega : y(x) < - \gamma^{-1/\alpha} \}, \\
    \Omega_\gamma^{+} &:= \{x\in \Omega : 0 < y(x) \leq \gamma^{-1/\alpha} \}, & \quad
    \Omega_\gamma^{-} &:= \{x\in \Omega : 0 > y(x) \geq - \gamma^{-1/\alpha} \}, \\
    \Omega_\gamma^0 &:= \{x\in \Omega: y(x) = 0\}
\end{alignedat}
\end{equation*}
as well as
\begin{equation*}
\begin{alignedat}{3}
    \Lambda_\gamma^+ &:= \{x\in \Omega: y_\gamma(x) > \tfrac{1}{2}\,\gamma^{-1/\alpha}\}, & \quad
    \Lambda_\gamma^- &:= \{x\in \Omega: y_\gamma(x) < -\tfrac{1}{2}\,\gamma^{-1/\alpha}\}, \\
    \Lambda_\gamma^0 &:= \{x\in \Omega: |y_\gamma(x)| \leq \tfrac{1}{2}\,\gamma^{-1/\alpha}\}.
\end{alignedat}
\end{equation*}
Note that 
\begin{equation*}
    \Omega_\gamma^{++} \cup \Omega_\gamma^{--} \cup \Omega_\gamma^{+} \cup 
    \Omega_\gamma^{-} \cup \Omega_\gamma^{0}  = \Omega \quad \text{and} \quad 
    \Lambda_\gamma^+ \cup \Lambda_\gamma^- \cup \Lambda_\gamma^0   = \Omega.
\end{equation*}
We denote the characteristic function of a set $M\subset \Omega$ by $\chi(M)$ and rewrite \eqref{eq:pdediff} by
\begin{equation*}
\begin{aligned}
     -\Delta e_\gamma &= 
     \Big(\chi(\Omega_\gamma^{++})\chi(\Lambda_\gamma^+)
    + \chi(\Omega_\gamma^{++})\chi(\Lambda_\gamma^0 \cup \Lambda_\gamma^-)\\
    & \quad + \chi(\Omega_\gamma^+)\chi(\Lambda_\gamma^+)
    + \chi(\Omega_\gamma^+)\chi(\Lambda_\gamma^0)
    + \chi(\Omega_\gamma^+)\chi(\Lambda_\gamma^-)\\
    & \quad + \chi(\Omega_\gamma^{--})\chi(\Lambda_\gamma^-)
    + \chi(\Omega_\gamma^{--})\chi(\Lambda_\gamma^0 \cup \Lambda_\gamma^+)\\
    & \quad + \chi(\Omega_\gamma^-)\chi(\Lambda_\gamma^+)
    + \chi(\Omega_\gamma^-)\chi(\Lambda_\gamma^0)
    + \chi(\Omega_\gamma^-)\chi(\Lambda_\gamma^-)\\
    & \quad + \chi(\Omega_\gamma^{0})\chi(\Lambda_\gamma^0)
    + \chi(\Omega_\gamma^{0})\chi(\Lambda_\gamma^+ \cup \Lambda_\gamma^-)\Big)( \beta_\gamma(y_\gamma) - q\big).
\end{aligned}
\end{equation*}
The idea is now to discuss every right-hand side separately. 
For this purpose define
\begin{equation*}
    b^{++}_{+} := \chi(\Omega_\gamma^{++})\chi(\Lambda_\gamma^+)( \beta_\gamma(y_\gamma) - q), 
    \quad
    b^{++}_{0-} := 
    \chi(\Omega_\gamma^{++})\chi(\Lambda_\gamma^0 \cup \Lambda_\gamma^-)( \beta_\gamma(y_\gamma) - q), 
\end{equation*}
and $b^+_+$, $b^+_0$, $b^+_-$, $b^{--}_-$, $b^{--}_{0+}$, $b^-_+$, $b^-_0$, $b^-_-$,
$b^0_0$, and $b^0_{+-}$ analogously.
Moreover, we denote the associated solutions of Poisson's equations by $e^{++}_+$ etc., i.e. 
\begin{equation*}
    e^{++}_+ \in H^1_0(\Omega), \quad -\Delta e^{++}_+ = b^{++}_+. 
\end{equation*}
Then, by superposition,  
\begin{equation*}
    e_\gamma = 
    e^{++}_+ +  e^{++}_{0-} + e^{+}_+ + e^{+}_0 + e^{+}_- + e^{--}_- + e^{--}_{0+} 
    + e^{-}_+ + e^{-}_0 + e^{-}_- + e^{0}_0 + e^{0}_{+-}. 
\end{equation*}
Now, according to the famous $L^\infty$-estimate of Stampacchia there is a constant $C>0$ independent of the right-hand side such that
\begin{equation*}
    e\in H^1_0(\Omega), \; -\Delta e = b
    \quad \Longrightarrow \quad \|e\|_{L^\infty(\Omega)} \leq C\, \|b\|_{L^s(\Omega)}
    \text{ with } s > \max\{d/2,1\}.
\end{equation*}
We will use this result to estimate (most of) the above errors. 

(a) Estimate of $e^{++}_+$: \\
On $\Omega_\gamma^{++} \cap \Lambda_\gamma^{+}$ we have $y > 0$ and $y_\gamma > \frac{1}{2}\gamma^{-1/\alpha}$. 
Therefore, $q = 1 = \sgn(y_\gamma)$ a.e.\ in $\Omega_\gamma^{++} \cap \Lambda_\gamma^{+}$ and thus,
\begin{equation*}
\begin{aligned}
    \|e^{++}_+\|_{L^\infty(\Omega)}
    & \leq C \,\|b^{++}_+\|_{L^s(\Omega)}\\
    & = C\, \Big( \int_{\Omega_\gamma^{++} \cap \Lambda_\gamma^{+}} 
   \big| \beta_\gamma(y_\gamma) - \sgn(y_\gamma)\big|^s\,\mathrm d x\Big)^{1/s}.
\end{aligned}
\end{equation*}
Now, thanks to $y_\gamma > \frac{1}{2}\gamma^{-1/\alpha}$, an estimate analogous to \eqref{eq:atanest} yields
\begin{equation*}
    |\beta_\gamma(y_\gamma) - \sgn(y_\gamma)|\leq \gamma^{1/\alpha - 1} 
    \quad \text{a.e.\ in }\Omega_\gamma^{++} \cap \Lambda_\gamma^{+}
\end{equation*}
and therefore,
\begin{equation*}
    \|e^{++}_+\|_{L^\infty(\Omega)} \leq C\, \gamma^{1/\alpha - 1} .
\end{equation*}
The estimates of $e_+^+$, $e_ -^{--}$ and $e_{-}^-$ can be done analogously.

(b) Estimate of $e^{++}_{0-}$:\\
On $M := \Omega_\gamma^{++} \cap (\Lambda_\gamma^{0} \cup \Lambda_\gamma^-)$ there holds
\begin{equation*}
    y(x) - y_\gamma(x) \geq \tfrac{1}{2}\,\gamma^{-1/\alpha}
\end{equation*}
and thus, the Cauchy-Schwarz inequality and \eqref{eq:L2est} give
\begin{equation*}
\begin{aligned}
    |M| = \gamma^{1/\alpha} \int_M \gamma^{-1/\alpha} \,\mathrm d x
    &\leq 2 \gamma^{1/\alpha} \int_M |y - y_\gamma|\, \mathrm d x \\
    & \leq  C\, \gamma^{1/\alpha}  \, |M|^{1/2}\,\|y -  y_\gamma\|_{L^2(\Omega)}
    \leq C \, |M|^{1/2}\,\gamma^{1/\alpha - 1/3}.
\end{aligned}
\end{equation*}
This yields
\begin{equation*}
\begin{aligned}
    \|e^{++}_{0-}\|_{L^\infty(\Omega)} 
    &\leq C\, \|b^{++}_{0-}\|_{L^s(\Omega)}\\
    &= C\, \Big(\int_M \big| \beta_\gamma(y_\gamma) - q\big|^s\,\mathrm d x\Big)^{1/s}\\
    &\leq 2C \, |M|^{1/s} \leq C \,\gamma^{2/s(1/\alpha - 1/3)}.
\end{aligned}    
\end{equation*}
The estimates of $e_-^+$, $e_{0+}^{--}$, $e_+^-$ and $e_{+-}^0$ can be done analogously.

(c) Estimate of $e^+_0$:\\
On $\Omega_\gamma^+ \cap \Lambda_\gamma^0$ there holds $|y|, |y_\gamma| \leq \gamma^{-1/\alpha}$.
Now, we argue as Nochetto in the proof of \cite[Theorem~2.1]{Nochetto1988} (cf. \eqref{eq:Nochetto}) to obtain
\begin{equation*}
    C\, \frac{2p+1}{(p+1)^2} \, \|e^+_0\|_{L^{2(p+1)}(\Omega)}^{2(p+1)} 
    \leq  \int_{\Omega_\gamma^+\cap \Lambda_\gamma^0} (\beta_\gamma(y_\gamma) - q) (e^+_0)^{2p+1} \mathrm d x
     \leq C\, (2\gamma)^{(-1/\alpha)(2p+1)}
\end{equation*}
and hence
\begin{equation*}
    \|e^+_0\|_{L^{2(p+1)}(\Omega)}
     \leq \Big(C\frac{(p+1)^2}{2p + 1} \, 2^{2p + 1}\, \big(\gamma^{-1/\alpha}\big)^{2p+1}\Big)^{1/(2p+2)}
    \overset{p\to \infty}{\longrightarrow} C \, \gamma^{-1/\alpha}.
\end{equation*}
This implies 
\begin{equation*}
    \|e^+_0\|_{L^{\infty}(\Omega)} \leq C\, \gamma^{-1/\alpha}.
\end{equation*}
The estimates of $e_0^-$ and $e_{0}^0$ can be done analogously.

In summary we deduce from (a)--(c) and the superposition of errors that 
\begin{equation*}
    \|e_\gamma\|_{L^\infty(\Omega)}
    \leq C\, \max\big\{ \gamma^{-1/\alpha}, \gamma^{1/\alpha -1}, \gamma^{2/s(1/\alpha - 1/3)}\big\}.
\end{equation*}
Next we adjust $\alpha$ so that the error contributions are equilibrated. Since $\alpha, s>1$, it is clear that 
$ \gamma^{2/s(1/\alpha - 1/3)} > \gamma^{1/\alpha -1}$ for $\gamma > 1$. Thus, the errors are equilibrated, if 
$-1/\alpha = 2/s(1/\alpha - 1/3)$, i.e. 
\begin{equation*}
    \alpha = \tfrac{3}{2}(2+s) > 
    \begin{cases}
        9/2, & \text{if}\ d\leq 2,\\
        21/4, & \text{if}\ d= 3.
    \end{cases}
\end{equation*}
We then finally arrive at 
\begin{equation*}
     \|e_\gamma\|_{L^\infty(\Omega)} \leq C\, \gamma^{-2/(3(2+s))}.
\end{equation*}

\end{proof}

\subsection{Error Estimates for the Discrete Regularized Variational Inequality}

The regularization procedure for the discrete variational inequality \eqref{disc_VI} is analogous to the continuous case (see Section \ref{sec_reg_VI}). Thus, the regularized problem of \eqref{disc_VI} is given by
\begin{equation}\label{reg_disc_VI}
y_{h,\gamma}\in V_h,\quad a(y_{h,\gamma},v_h)+(\beta_\gamma(y_{h,\gamma}),v_h)=\langle u,v_h\rangle\quad \forall v_h\in V_h,
\end{equation}
where $V_h$ is the space of piecewise continuous functions of Section \ref{sec:Discretization}.
An inspection of the proof of Theorem~\ref{theo_conv_to_original_sol} reveals that its arguments are not limited 
to the continuous case but also valid in the discrete case. 
Note in this context that, for fixed $h>0$, $V_h$ is a closed subspace of $H^1_0(\Omega)$ such that the weak 
limit of a weakly convergent sequence $\{y_{h, \gamma}\}_{\gamma > 0} \subset V_h$ 
(for $\gamma \to \infty$) is again an element of $V_h$.
Together with the arguments in the proof of Theorem~\ref{theo_conv_to_original_sol} this shows the following:

\begin{proposition}\label{prop:properties_reg_disc_VI}
    Assume that a sequence $\{u_\gamma\} \subset H^{-1}(\Omega)$ is given such that $u_\gamma \to u$ in $H^{-1}(\Omega)$
as $\gamma \to \infty$.   
    Let us denote the solutions of \eqref{reg_disc_VI} associated with $u_\gamma$ by $y_{h,\gamma}$ 
    and the solution of \eqref{disc_VI} by $y_h$.
    Then $y_{h,\gamma}\to y_h$ in $H_0^1(\Omega)$ as $\gamma\to \infty$.
\end{proposition}

In the following we set $e_h:=y_\gamma-y_{h,\gamma}$. Furthermore, we introduce a regularization of the Dirac measure which fulfills
\begin{equation}\label{def_Dirac}
\delta\in C_0^\infty(\Omega),\quad \delta\geq 0,\quad \int_\Omega \delta\, \mathrm dx=1,\quad
\supp\delta\subset B_{\epsilon h}(x_0):=\{x\in\Omega : |x-x_0|<\epsilon h\},
\end{equation}
where $x_0\in\Omega$ and $\epsilon>0$. Subtracting \eqref{reg_disc_VI} from \eqref{eq:pdegamma} leads to the equation
\begin{equation}\label{eq_error_equation}
a(e_h,v_h)+(be_h,v_h)=0\quad \forall v_h\in V_h,
\end{equation}
where
\begin{equation*}
b(x):=\begin{cases}
[\beta_\gamma(y_\gamma)-\beta_\gamma(y_{h,\gamma})]/e_h &\text{if}\ e_h(x)\neq 0\\
0 & \text{if}\ e_h(x)=0.
\end{cases}
\end{equation*}

The dual problem of \eqref{eq_error_equation} is given by
\begin{equation*}
G\in H_0^1(\Omega):\quad -\Delta G+ bG=\delta.
\end{equation*}

$G$ is the regularized Green's function and satisfies the bound

\begin{equation}\label{eq_L1_boundedness}
\|\Delta G\|_{L^1(\Omega)}\leq C
\end{equation}

(cf. \cite{Nochetto1988}). We define the Ritz projection $G_h$ of the regularized Green's function $G$ onto $V_h$ by
\begin{equation*}
G_h\in V_h,\quad a(G, v_h)=a(G_h, v_h)\quad \forall v_h\in V_h.
\end{equation*}

\begin{lemma}\label{lemma_G}
It holds
\begin{equation*}
\|G-G_h\|_{L^{p'}(\Omega)}\leq ch^{2-d/p}|\log h|,
\end{equation*}
where $p'$ denotes the conjugate exponent to $p$.
\end{lemma}

\begin{proof}
Let $\phi\in L^\infty(\Omega)$ be given and define $\psi\in H_0^1(\Omega)$ by
\begin{equation*}
-\Delta \psi=\phi\quad \text{in}\ H^{-1}(\Omega).
\end{equation*}
Since $\Omega$ is assumed to be $W^{2,p}$-regular, we have
\begin{equation*}
\|\psi\|_{W^{2,p}(\Omega)}\leq C\|\phi\|_{L^p(\Omega)}.
\end{equation*}
By the last equation in \cite[p.~249]{Nochetto1988} and the $L^1$-boundedness of $\Delta G$ (cf. \eqref{eq_L1_boundedness}) we deduce
\begin{align*}
(G-G_h,\phi)&=(-\Delta G,\psi-\psi_h)\\
&\leq \|\Delta G\|_{L^1(\Omega)}\|\psi-\psi_h\|_{L^\infty(\Omega)}\\
&\leq C \|\psi-\psi_h\|_{L^\infty(\Omega)},
\end{align*}
where $\psi_h$ is the Ritz projection of $\psi$. Further by \cite[p.~3]{SchatzWahlbin1982} we know
\begin{equation*}
\|\psi-\psi_h\|_{L^\infty(\Omega)}\leq ch^{2-d/p}|\log h| \|\psi\|_{W^{2,p}(\Omega)}.
\end{equation*}
Plugging everything together we arrive at
\begin{equation*}
(G-G_h,\phi)\leq Ch^{2-d/p}|\log h| \|\phi\|_{L^p(\Omega)}.
\end{equation*}
Hence, we finally obtain
\begin{equation*}
\|G-G_h\|_{L^{p'}(\Omega)}=\sup_{\phi\neq 0}\frac{(G-G_h,\phi)}{\|\phi\|_{L^p(\Omega)}}\leq ch^{2-d/p}|\log h|.
\end{equation*}
\end{proof}

\begin{theorem}\label{theo_disc_reg_error}
Let $u\in L^p(\Omega)$, $p\in(1,\infty)$. Then the discretization error of the regularized variational inequality can be bounded by
\begin{equation*}
\|e_h\|\leq Ch^{2-d/p}|\log h|(\|u\|_{L^p(\Omega)}+1).
\end{equation*}

\end{theorem}

\begin{proof}
We trace the proof of \cite[Theorem~2.3]{Nochetto1988}. Let $x_0\in\Omega$ be such that $e_h(x_0)=\|e_h\|_{L^\infty(\Omega)}$ and let $x_1\in B_{\epsilon h}(x_0)$ be such that $(e_h,\delta)=e_h(x_1)$, where $\delta$ is defined in \eqref{def_Dirac}. As in the proof of \cite[Theorem~2.3]{Nochetto1988} we obtain by the mean value theorem and an inverse inequality 
\begin{equation*}
\|e_h\|_{L^\infty(\Omega)}=e_h(x_0)\leq e_h(x_1)+C\epsilon\|I_he_h\|_{L^\infty(\Omega)}+\epsilon h\|\nabla (y_\gamma-I_hy_\gamma)\|_{L^\infty(\Omega)},
\end{equation*}
where $I_h: C(\bar{\Omega})\to V_h$ denotes the Lagrange interpolation operator. Remember that $\nabla y_\gamma$ and $\nabla y_{h,\gamma}\in L^\infty(\Omega)$ which implies the continuity of $e_h$ and thus, the Lagrange interpolant is well-defined. For the last term we apply a standard Lagrange interpolation error estimate which yields
\begin{equation*}
\|\nabla (y_\gamma-I_hy_\gamma)\|\leq Ch^{1-d/p}\|y_\gamma\|_{W^{2,p}(\Omega)}\leq Ch^{1-d/p}(1+\|u\|_{H^1(\Omega)})
\end{equation*}
independent of $\gamma$. The term involving $\epsilon$ can be compensated by the left-hand side by using the stability of $I_h$ in $L^\infty(\Omega)$ and choosing $\epsilon$ sufficiently small. Hence, it remains to estimate $e_h(x_1)$. Completely analogous to the last equation in the proof of \cite[Theorem~2.3]{Nochetto1988} we deduce
\begin{equation*}
e_h(x_1)=(u-\beta_\gamma(y_{h,\gamma}),G-G_h)\leq C(\|u\|_{L^p(\Omega)}+1)\|G-G_h\|_{L^{p'}(\Omega),}
\end{equation*}
where we used the uniform boundedness of $\beta_\gamma$ independent of $\gamma$ and its input. Lemma \ref{lemma_G} finally yields the desired estimate.
\end{proof}

\subsection{$\boldsymbol{L^\infty}$-Error Estimates for the Original Problem}

Before we turn to a priori error estimates in $L^\infty(\Omega)$, let us shortly address a convergence result in energy spaces 
needed for Lemma~\ref{lemma_strong_convergence}. 

\begin{lemma}\label{lem:convH1}
    Let $\{u_h\}_{h>0}$ be a sequence that converges strongly in $H^{-1}(\Omega)$ to $u \in H^{-1}(\Omega)$ as $h\searrow 0$.
    Denote the solution of the discretized VI in \eqref{disc_VI} associated with $h$ and $u_h$ by $y_h$ and the solution of 
    \eqref{VI} corresponding to $u$ by $y$. Then $y_h \to y$ in $H^1_0(\Omega)$ as $h\searrow 0$.
\end{lemma}

\begin{proof}
    Let $v_h \in V_h$ be fixed, but arbitrary. We test \eqref{VI} with $y_h$ and \eqref{disc_VI} with $v_h$ and add the arising 
    inequalities to obtain
    \begin{equation*}
    \begin{aligned}
        a(y - y_h, y - y_h)
        &\leq a(y_h, v_h - y) + \|v_h\|_{L^1(\Omega)} - \|y\|_{L^1(\Omega)} 
        + \dual{u - u_h}{y-y_h} + \dual{u_h}{y - v_h}\\
        &\leq C\big( \|y_h\|_{H^1(\Omega)} +1 + \|u_h\|_{H^{-1}(\Omega)}\big) \|v_h - y\|_{H^1(\Omega)}\\
        &\quad + \big(\|y\|_{H^1(\Omega)} + \|y_h\|_{H^1(\Omega)}\big) \|u - u_h\|_{H^{-1}(\Omega)}.
    \end{aligned}
    \end{equation*}
    By testing \eqref{disc_VI} with zero we see that $\|y_h\|_{H^1(\Omega)}$ is bounded by a constant independent of $h$.
    Since $v_h$ was arbitrary, the coercivity of the bilinear form  implies
    \begin{equation*}
        \|y-y_h\|_{H^1(\Omega)}^2 \leq C \Big(\|u - u_h\|_{H^{-1}(\Omega)} + \inf_{v_h \in V_h} \|v_h - y\|_{H^1(\Omega)}\Big).
    \end{equation*}
    The convergence of $u_h$ by assumption and the density of $\bigcup_{h>0}V_h$ in $H^1_0(\Omega)$ then implies the result.
\end{proof}

Applying the results of the previous subsections we are able to derive the following discretization error for the states.
\begin{theorem}\label{theo_Nochetto2}
If $u\in L^p(\Omega)$, $p\in (1,\infty)$, then there exists a constant $C>0$ such that
\begin{equation*}
\|y-y_h\|_{L^\infty(\Omega)}\leq C|\log(h)|h^{2-d/p}(\|u\|_{L^p(\Omega)}+1).
\end{equation*}
If $u\in L^\infty(\Omega)$, then there exists a constant $C>0$ such that
\begin{equation*}
\|y-y_h\|_{L^\infty(\Omega)}\leq C(h|\log(h)|)^2(\|u\|_{L^\infty(\Omega)}+1).
\end{equation*}
$C$ is independent of $h$.
\end{theorem}

\begin{proof}
We consider the following error splitting
\begin{equation}\label{eq_error_splitting}
\|y-y_{h}\|_{L^\infty(\Omega)}\leq \|y-y_\gamma\|_{L^\infty(\Omega)}+\|y_\gamma-y_{h,\gamma}\|_{L^\infty(\Omega)} + \|y_{h,\gamma}-y_{h}\|_{L^\infty(\Omega)}
\end{equation}

and observe that $y_h$, $y_{h,\gamma}\in V_h$, which is a finite-dimensional space. Since for fixed $h$ all norms are equivalent in $V_h$ the strong convergence $y_{h,\gamma}\to y_h$ in $H_0^1(\Omega)$ (cf. Proposition \ref{prop:properties_reg_disc_VI}) implies
\begin{equation}\label{strong_Linf_conv}
y_{h,\gamma}\to y_h\quad \text{in}\ L^\infty(\Omega)\ \text{as}\ \gamma\to \infty.
\end{equation}
The choice $\gamma = Ch^{-3(2+s)}$ with $s>\max\{d/2, 1\}$ in Theorem \ref{theo_reg_error} and Theorem \ref{theo_disc_reg_error}  complete the proof.

\end{proof}

\section{Regularized Optimal Control  Problem}\label{sec:Reg_OCP}

Let a locally optimal solution $\bar u \in L^2(\Omega)$ of \eqref{prob_P} be given. Then
we consider the following regularized optimal control problem:

\begin{equation}\label{prob_P_gamma}\tag{P$_\gamma$}
	\left.\begin{aligned}
	& 
	\min_{(y_\gamma, u_\gamma) \in H_0^1(\Omega)\times L^2(\Omega)} J(y_\gamma,u_\gamma)
	+ \|u_\gamma - \bar u\|_{L^2(\Omega)}^2,\\
	& \text{s.t.}\quad a(y_\gamma, v)+\frac{2}{\pi}\int_\Omega\arctan(\gamma y_\gamma) v\ \mathrm dx =\langle u_\gamma, v\rangle\quad \forall v\in H_0^1(\Omega).
	\end{aligned}
	\right\}
\end{equation}
The additional penalty term involving $\bar u$ will ensure the convergence to the given local solution 
as $\gamma \to \infty$, a technique that goes back to \cite{Barbu1984}, cf.\ also \cite{MignotPuel1984}.

Let us introduce the control-to-state operator $S_\gamma: L^2(\Omega)\to H_0^1(\Omega)$ which assigns $u_\gamma\in L^2(\Omega)$ to the solution $y_\gamma\in H_0^1(\Omega)$ of the PDE in \eqref{prob_P_gamma}
(which exists due to the monotonicity of $\arctan$).
As a direct consequence of the Fr\'echet differentiability of $\arctan$ from $H^1_0(\Omega)$ to $L^2(\Omega)$ and 
the implicit function theorem we obtain the following result:

\begin{lemma}\label{lemma_diff_S_gamma}
$S_\gamma$ is continuously Fr\'{e}chet differentiable from $L^2(\Omega)$ to $H_0^1(\Omega)$ and its derivative $w=S_\gamma'(u_\gamma)h$ at $u_\gamma \in L^2(\Omega)$ in direction $h\in L^2(\Omega)$
is given by the solution of the linearized equation
\begin{equation}\label{der_S_gamma}
a(w,v) + \int_\Omega\frac{2\gamma}{\pi(1+\gamma^2y_\gamma^2)}wv\ \mathrm dx= (h,v)\quad \forall v\in H_0^1(\Omega),
\end{equation}
where $y_\gamma = S_\gamma(u_\gamma)$.
\end{lemma}

The next theorem covers the approximation of $\bar u$ by locally optimal solutions of \eqref{prob_P_gamma}:

\begin{theorem}\label{theo_strong_conv_P_gamma}
There exists a sequence of local solutions to \eqref{prob_P_gamma}, that converges strongly to $\bar u$ as $\gamma \to \infty$.
\end{theorem}

\begin{proof}
Since the proof combines well-known arguments, we only sketch it. 
Although our regularization differs from the one in 
\cite{DelosReyes2011}, our findings from Appendix~\ref{sec_reg_VI}, in particular Theorem~\ref{theo_conv_to_original_sol}, 
allow to readily transfer 
the proofs of \cite[Theorem~3.4 resp.~3.5]{DelosReyes2011} to our setting.
These theorems cover the convergence of global minimizers, to be more precise, for every $\gamma > 0$ 
there is a (globally optimal) solution of \eqref{prob_P_gamma} and, as $\gamma\to \infty$ a 
subsequence of these converges to a global minimizer of \eqref{prob_P} with the objective functional replaced 
by the one in \eqref{prob_P_gamma}, i.e. 
\begin{equation}\label{eq:Ppenalty}
    \min_{u \in L^2(\Omega)} \; J(S(u), u) + \|u - \bar u\|_{L^2(\Omega)}^2,
\end{equation}
where $S$ denotes again the solution operator of the VI in \eqref{VI} from Lemma~\ref{lemma_Lipschitz}.
Moreover, a classical localization argument (cf.\ e.g.\ \cite{CasasTroeltzsch2002})
implies that the global convergence theory carries over to isolated local minimizers.
However, due to the additional penalty term in \eqref{eq:Ppenalty}, it is easy to see that $\bar u$ is an isolated 
local minimizer of \eqref{eq:Ppenalty}, which implies the result.
\end{proof}

\begin{theorem}\label{theo_reg_opt_system}
If $(\bar{y}_\gamma, \bar{u}_\gamma)$ is a locally optimal solution of the regularized problem \eqref{prob_P_gamma}, then there exist $\bar{p}_\gamma\in H_0^1(\Omega)$ and $\bar\mu_\gamma \in L^2(\Omega)$ so that the following optimality system is satisfied:
\begin{subequations}
\begin{align}
&a(\bar{y}_\gamma, v)+\frac{2}{\pi}(\arctan(\gamma\bar{y}_\gamma),v) =\langle\bar{u}_\gamma, v\rangle\quad\forall v\in H_0^1(\Omega)\label{reg_opt_system_1}\\
&a(\bar{p}_\gamma, v)+(\bar{\mu}_\gamma, v)= (\bar{y}_\gamma -y_d,v)\quad\forall v\in H_0^1(\Omega)\label{reg_opt_system_2}\\
& \bar{\mu}_\gamma =\frac{2\gamma}{\pi(1+\gamma^2 \bar{y}_\gamma^2)}\bar{p}_\gamma
\quad \text{a.e.\ in }\Omega \nonumber\\
&(\bar{\mu}_\gamma,\bar{p}_\gamma)\geq 0\label{reg_opt_system_3}\\
& \bar p_\gamma + \nu (\bar{u}_\gamma-u_d)  + 2(\bar u_\gamma - \bar u) = 0 \quad \text{a.e.\ in }\Omega \label{reg_opt_system_4}
\end{align}
\end{subequations}
\end{theorem}

\begin{proof}
We introduce the reduced functional $f_\gamma: L^2(\Omega)\to \R$, which is defined by
\begin{equation*}
f_\gamma(u_\gamma)=J(S_\gamma(u_\gamma),u_\gamma) +\|u_\gamma- \bar u\|_{L^2(\Omega)}^2.
\end{equation*}
Due to the local optimality of $\bar u_\gamma$ and the differentiability of $S_\gamma$ we obtain
\begin{equation*}
f_\gamma'(\bar{u}_\gamma)h=0\quad\forall h\in L^2(\Omega).
\end{equation*}
Now let $h\in L^2(\Omega)$ be arbitrary and define $w = S_\gamma'(\bar u_\gamma)h$.
Introducing $\bar{p}_\gamma$ as the unique solution of the adjoint equation \eqref{reg_opt_system_2} (which is clearly uniquely solvable by the Lax-Milgram lemma) and using the definition of $\bar{\mu}_\gamma$ we arrive at
\begin{align*}
f_\gamma'(\bar{u}_\gamma)h&=(\bar{y}_\gamma-y_d, w)+(\nu(\bar{u}_\gamma-u_d) + 2(\bar u_\gamma - \bar u), h)\\
&= a(\bar{p}_\gamma, w) + (\bar{\mu}_\gamma, w) + (\nu(\bar{u}_\gamma-u_d) + 2(\bar u_\gamma - \bar u),h)\\
&=a(\bar{p}_\gamma, w)+\left(\frac{2\gamma}{\pi(1+\gamma^2\bar{y}_\gamma^2)}\bar{p}_\gamma,w\right)
+(\nu (\bar{u}_\gamma -u_d) + 2(\bar u_\gamma - \bar u),h).
\end{align*}
Hence, by Lemma \ref{lemma_diff_S_gamma} it follows that 
$(\bar{p}_\gamma+\nu (\bar{u}_\gamma - u_d) + 2(\bar u_\gamma - \bar u), h)=0$
and, since $h$ was arbitrary, this implies the result.
\end{proof}

\begin{lemma}\label{lemma_pgamma_bounded}
Consider the sequence of local solutions from Theorem~\ref{theo_strong_conv_P_gamma} and denote the associated
sequence of adjoint states from Theorem \ref{theo_reg_opt_system} by $\{ \bar p_\gamma\}_{\gamma > 0}$.
Then $\{\bar{p}_\gamma\}$ is bounded in $H_0^1(\Omega)$.
\end{lemma}

\begin{proof}
Due to the coercivity of $a(\cdot, \cdot)$ we get by testing \eqref{reg_opt_system_2} with $\bar{p}_\gamma\in H_0^1(\Omega)$
\begin{align*}
c\|\bar{p}_\gamma\|_{H^1(\Omega)}^2&\leq a(\bar{p}_\gamma,\bar{p}_\gamma)
= -\langle\bar{\mu}_\gamma, \bar{p}_\gamma\rangle + (\bar{y}_\gamma -y_d, \bar{p}_\gamma)
\leq \left(\|\bar{y}_\gamma\|_{L^2(\Omega)} + \|y_d\|_{L^2(\Omega)}\right)\|\bar{p}_\gamma\|_{H^1(\Omega)},
\end{align*}
where we used \eqref{reg_opt_system_3}. Thus, we obtain
$\|\bar{p}_\gamma\|_{H^1(\Omega)}\leq\frac{1}{c}(\|\bar{y}_\gamma\|_{L^2(\Omega)} + \|y_d\|_{L^2(\Omega)})$.
Using the essential boundedness of $\beta_\gamma$ independent of $\gamma$ and a bootstrapping 
argument as in the proof of Lemma~\ref{lemma_reg_VI_regularity}, we see that 
$\|\bar y_\gamma\|_{H^1(\Omega)} \leq C (1 + \|\bar u_\gamma\|_{L^2(\Omega)})$ with a constant $C>0$ 
independent of $\gamma$. Since $\{\bar u_\gamma\}$ converges by assumption and is therefore bounded,  
this implies the claim.
\end{proof}

\end{appendices}

\bibliographystyle{plain}
\bibliography{biblio_VI.bib}{}
\end{document}